\documentclass[a4paper,10pt]{amsart}

\usepackage[utf8]{inputenc}
\usepackage[T1]{fontenc}
\usepackage{amssymb}
\usepackage[english]{babel}
\usepackage{mathrsfs}
\usepackage{tikz-cd}
\usepackage{verbatim}
\usepackage{color}
\usepackage{multirow}
\usepackage[shortlabels]{enumitem}
\setlist[enumerate]{label=\rm{(\arabic*)}}
\setlist[enumerate,2]{label=\rm{(\roman*)}}
\setlist[itemize]{label=\raisebox{0.25ex}{\tiny$\bullet$}}
\usepackage[backref, colorlinks, linktocpage, citecolor = blue, linkcolor = blue]{hyperref}

\usepackage[left=2.5cm,right=2.5cm,top=2.5cm,bottom=2.5cm]{geometry}

\theoremstyle{plain}
\newtheorem{theorem}{Theorem}[section]
\newtheorem{corollary}[theorem]{Corollary}
\newtheorem{proposition}[theorem]{Proposition}
\newtheorem{lemma}[theorem]{Lemma}
\newtheorem*{question}{Question}
\newtheorem*{acknowledgments}{Acknowledgments}
\theoremstyle{definition}
\newtheorem{definition}[theorem]{Definition}
\newtheorem{remark}[theorem]{Remark}

\theoremstyle{plain}
\newtheorem{theoremA}{Theorem}

\def\AA{\mathbb{A}}
\def\PP{\mathbb{P}}
\def\kk{\Bbbk}
\def\ZZ{\mathbb{Z}}
\def\NN{\mathbb{N}}
\def\GG{\mathbb{G}}

\def\E{\mathsf{E}}
\def\M{\mathsf{M}}
\def\G{\mathsf{G}}
\def\H{\mathsf{H}}
\def\V{\mathsf{V}}
\def\C{\mathsf{C}}
\def\h{\mathsf{h}}
\def\D{\mathsf{D}} 
\def\O{\mathcal{O}}
\def\A{\mathscr{A}}
\def\s{\sigma}
\def\F{\mathsf{F}}
\def\S{\mathsf{S}}
\def\T{\mathsf{T}}
\def\X{\mathsf{X}}
\def\Y{\mathsf{Y}}
\def\L{\mathsf{L}}
\def\EF{\mathcal{E}}

\def\g{\mathsf{g}}
\def\seg{\mathfrak{S}}
\def\div{\operatorname{div}}
\def\Bs{\operatorname{Bs}}
\def\bup{\mathsf{Bl}}
\def\Spec{\operatorname{Spec}}

\def\PGL{\mathsf{PGL}}
\def\Aut{\mathsf{Aut}}
\def\AutC{\mathsf{Aut_{\sf{C}}}}
\def\Autinfini{\mathsf{Aut}_{\infty}}
\def\Autzero{\mathsf{{Aut}^\circ}}
\def\Bir{\mathsf{Bir}}

\begin{document}
\title{Connected algebraic groups acting on algebraic surfaces}
\author{Pascal Fong}
\address{Universit\"at Basel, Departement Mathematik und Informatik, Spiegelgasse 1, CH--4051 Basel, Switzerland}
\email{pascal.fong@unibas.ch}
\thanks{The author acknowledges support by the Swiss National Science Foundation Grant “Curves in the space” 200021–169508.}

\maketitle

\begin{abstract}
	We classify the maximal connected algebraic subgroups of $\Bir(\X)$, when $\X$ is a surface.
\end{abstract}

\section{Introduction}
	
In this text, varieties are reduced and separated schemes of finite type over an algebraically closed field $\kk$. Unless otherwise stated, curves are also assumed to be smooth, irreducible and projective. If $\X$ is a variety, we denote by $\Bir(\X)$ the group of birational transformations of $\X$. A subgroup $\G\subset \Bir(\X)$ is \emph{algebraic} if there exists a structure of algebraic group (i.e.\ a smooth group scheme of finite type) on $\G$ such that the action $\G \times \X \dashrightarrow \X$ induced by the inclusion of $\G$ into $\Bir(\X)$ is a rational action (see Definition \ref{actions}). Moreover, $\G$ is a \emph{maximal algebraic subgroup} of $\Bir(\X)$ if there is no algebraic subgroup $\G'$ of $\Bir(\X)$ which strictly contains $\G$. If $\X$ is a projective variety, then the subgroup $\Aut(\X) \subset \Bir(\X)$ of automorphisms of $\X$ is a group scheme (by \cite{Matsumara_Oort}); and the connected component of the identity $\Autzero(\X)$ is an algebraic subgroup of $\Bir(\X)$. In this paper we answer in Theorem \ref{D} the following question when $\operatorname{char}(\kk)=0$.
	
\begin{question}
	What are the maximal connected algebraic subgroups of $\Bir(\X)$, when $\X$ is an algebraic surface (or equivalently, when $\X$ is a projective surface)?
\end{question}	

The maximal connected algebraic subgroups of $\Bir(\PP^2)$ have been studied by Enriques in \cite{Enriques}: the maximal connected algebraic subgroups of $\Bir(\PP^2)$ are conjugate to $\Autzero(\PP^2)$ or to $\Autzero(\mathbb{F}_n)$ for some $n\in \NN\setminus \{1\}$, where $\mathbb{F}_n$ denotes the $n$-th \emph{Hirzebruch surface} (i.e.\ the $\PP^1$-bundle over $\PP^1$ having a section of self-intersection $-n$). Furthermore, any connected algebraic subgroup of $\Bir(\PP^2)$ is contained in a maximal connected algebraic subgroup. We will show in Theorem \ref{A} that if $\S$ is a $\PP^1$-bundle over a curve $\C$ of genus $\g \geq 1$, it is not always true that $\Autzero(\S)$ is contained in a maximal connected algebraic subgroup of $\Bir(\S)$. Our approach to prove Theorem \ref{A} uses elementary tools like blowups and contractions, and a classification of automorphisms of ruled surfaces due to Maruyama in \cite{Maruyama}.

\begin{theoremA}\label{A}
	Let $\S$ be a non trivial $\PP^1$-bundle over a curve $\C$ of genus $\g$. We assume that $\g\geq 2$, or that $\g=1$ and $\S$ admits a section of negative self-intersection number. Then there exists a family $(\S_n)_{n\geq 1}$ of $\PP^1$-bundles over $\C$ with birational maps $\phi_n\colon\S\dashrightarrow \S_n$ such that:
	\[
	\Autzero(\S) \subset \phi_1^{-1} \Autzero(\S_1) \phi_1 \subset ... \subset \phi_n^{-1} \Autzero(\S_n) \phi_n \subset ...
	\]
	is not a stationary sequence. In particular, the connected algebraic subgroup $\Autzero(\S)$ of $\Bir(\S)$ is not maximal.
\end{theoremA}

Then we study the connected algebraic subgroups of $\Bir(\C\times \PP^1)$ when $\C$ is a curve of genus $1$. So assume in this paragraph that $\C$ is an elliptic curve. We denote the Atiyah's ruled surfaces by $\A_{0,\C}$ and $\A_{1,\C}$ (see Theorem \ref{Atiyah}) and if $z_1,z_2\in \C$ are distinct points, we denote by $\S_{z_1,z_2}$ the ruled surface $\PP(\O_\C(z_1)\oplus \O_\C(z_2))$. A geometrical description of these surfaces via sequences of blowups and contractions from $\C\times \PP^1$ is given in Section \ref{ruledelemtrans}. Then we show that their automorphism groups are maximal connected algebraic subgroups. With Theorem \ref{A}, this leads to Theorem \ref{B}:

\begin{theoremA}\label{B}
	Let $\C$ be a curve of genus $\g$ and $\G$ be a maximal connected algebraic subgroup of $\Bir(\C\times \PP^1)$. If $\g\geq 2$ then $\G$ is conjugate to the maximal algebraic subgroup $\Autzero(\C\times \PP^1)$, and if $\g= 1$ then $\G$ is conjugate to one of the following:
	\begin{enumerate}
		\item $\Autzero(\C\times \PP^1)$,
		\item $\Autzero(\S_{z_1,z_2})$ where $z_1$ and $z_2$ are distinct points in $\C$,
		\item $\Autzero(\A_{0,\C})$,
		\item $\Autzero(\A_{1,\C})$. 
	\end{enumerate}
	The algebraic subgroups in $(1), (2), (3),(4)$ are all maximal and are pairwise not conjugate. Moreover in case $(2)$, two algebraic subgroups $\Autzero(\S_{z_1,z_2})$ and $\Autzero(\S_{z_1',z_2'})$ are conjugate if and only if there exists $f\in \Aut(\C)$ such that $f(\{z_1,z_2\})=\{z_1',z_2'\}$. 
\end{theoremA}

We have $\Autzero(\C\times \PP^1) \simeq \Autzero(\C)\times \PGL_2$, which is isomorphic to $\C\times \PGL_2$ if $\g=1$, or isomorphic to $\PGL_2$ if $\g\geq 2$. Hence the structure of $\Autzero(\C\times \PP^1)$ as algebraic group is simple to understand. In Theorem \ref{Maruyamaext}, we describe the other maximal connected algebraic groups of Theorem \ref{B} as extensions of an elliptic curve by a linear group. The structures of $\Autzero(\A_{0,\C})$ and $\Autzero(\A_{1,\C})$ as extensions in Theorem \ref{Maruyamaext} are actually a direct consequence of Maruyama's theorem, and it has already been proven in a more general setting in \cite[Theorem 4.2, 2.(b) and 2.(c)]{Laurent}. However, our approach only uses elementary techniques of birational geometry to compute the kernel of the morphism $\Autzero(\S) \to \Autzero(\C)$ induced by Blanchard's Lemma (\ref{Blanchard}), when $\S$ is isomorphic to $\A_{0,\C}$ or $\S_{z_1,z_2}$ for some $z_1,z_2\in \C$. Moreover, we prove the extension of $\Autzero(\A_{1,\C})$ by giving an explicit construction of the surface $\A_{1,\C}$. 

Combining Theorems \ref{A} and \ref{B} with general arguments from the theory of algebraic groups, we show the following equivalence:

\begin{theoremA}\label{C}
	Let $\X$ be a surface. Then every connected algebraic subgroup of $\Bir(\X)$ is contained in a maximal one if and only if $\X$ is not birationally equivalent to $\C\times \PP^1$ for some curve  $\C$ of genus $\g\geq 1$.
\end{theoremA}

Finally, we answer the question when the characteristic of $\kk$  is $0$, by giving the classification of all maximal connected algebraic subgroups in dimension $2$ (Theorem \ref{D}). If the characteristic is positive, we have a partial classification, see Proposition \ref{partialclassification} and remark \ref{obstructionclassificationpositivecar}.

\begin{theoremA}\label{D}
Let $\X$ be a surface over a field $\kk$ of characteristic $0$. We denote by $E$ the set of  surfaces of the form $(\C\times \Y)/\F$ where $\C$ is an elliptic curve, $\Y$ is a smooth curve of general type, and $\F$ is a finite subgroup of $\Autzero(\C)$ acting diagonally on $\C\times \Y$. The pairs $(\X,\Autzero(\X))$ are classified as following:
\begin{table}[ht]
	\begin{center}
		\begin{tabular}{| c |  p{3.5cm} | p{10cm} |}
			\hline
			$\kappa(\X)$  &	Representative of the birational class of $\X$ & $\Autzero(\X)$ \\
			\hline
			\hline
			
			\multirow{2}{*}{}
			& Rational surface & Maximal if and only if $\X$ is isomorphic to $\PP^2$ or $\mathbb{F}_n$ with $n\neq 1$. Else $\Autzero(\X)$ is conjugate to an algebraic subgroup of a maximal one.\\ \cline{2-3}
			$-\infty$ & Ruled surface (over a curve $\C$ of positive genus)& Maximal if  and only if  $\X$ is isomorphic to $\C\times \PP^1$, or $\A_{0,\C}$, or $\A_{1,\C}$, or  $\S_{z_1,z_2}$ with $z_1,z_2\in \C$ (the three last cases happen only when $\C$ is an elliptic curve). Else $\Autzero(\X)$ is not maximal and fits into an infinite chain of strict inclusions.\\
			\hline
			
			\multirow{4}{*}{}
			 & Abelian surface & $\Autzero(\X)\simeq \X$ if and only if $\X$ is an abelian surface; and in this case $\Autzero(\X)$ is maximal. Else, $\Autzero(\X)$ is trivial and is not maximal.\\ \cline{2-3} & $K3$ surface & $\Autzero(\X)$ trivial and maximal.\\ \cline{2-2} &  Enriques surface  &  \\ \cline{2-3} 
			 $0$  &  Bielliptic surface &  $\Autzero(\X)\simeq \C$ is an elliptic curve if and only if $\X\simeq (\C\times \Y)/\F$ where $\C,\Y$ are elliptic curves and $\F$ is a finite group acting on $\C$ by translations, and acting also on $\Y$ not only by translations (equivalently, $\Y/\F\simeq \PP^1$ and $\X$ is a bielliptic surface). In this case, $\Autzero(\X)$ is maximal. Else, $\Autzero(\X)$ is trivial and is not maximal. \\ \hline

			$1$ & Properly elliptic surface & $\Autzero(\X)\simeq \C$ is an elliptic curve if and only if $\X\simeq (\C\times \Y)/\F$ where $\Y$ is a smooth curve of general type and $\F$ is a finite group acting on diagonally on $\C\times \Y$ and by translations on $\C$ (i.e.\ $\X\in E$). In this case, $\Autzero(\X)$ is maximal. 
			
			If $\X$ is birational to an element of $E$ but $\X\notin E$, then $\Autzero(\X)$ is trivial and not maximal. Else $\Autzero(\X)$ is trivial and maximal. \\ \hline
			
			$2$ & Surface of general type & $\Autzero(\X)$ trivial and maximal. \\ \hline		
		\end{tabular}
	\end{center}
\end{table}

\end{theoremA}

\begin{acknowledgments}
	I am thankful to Jérémy Blanc, Michel Brion, Serge Cantat, Gabriel Dill, Bruno Laurent, Julia Schneider, Ronan Terpereau, Immanuel Van Santen, Egor Yasinsky, Sokratis Zikas, ShengYuan Zhao for interesting discussions and ideas, and to the referee for the careful lecture and many useful remarks.
\end{acknowledgments}

\section{Preliminaries}

\subsection{Equivariance and maximality}

In this section we reduce the question to the maximality of the automorphism groups of \emph{minimal surfaces}, i.e.\ smooth projective surfaces without $(-1)$-curves. The idea has already been used in the rational case to study algebraic subgroups of $\Bir(\PP^2)$ (e.g.\ see \cite{BlancCremona}, and \cite{RobayoZimmermann} when $\kk=\mathbb{R}$).

\begin{definition}\label{actions}
	Let $\G$ be an algebraic group, $\X$ be a variety and $\alpha\colon\G\to \Bir(\X)$ be a group homomorphism.
	\begin{enumerate}
		\item The map $\alpha$ is a \emph{rational action} of $\G$ on $\X$ if there exists an open $U$ of $\G\times \X$ such that:
		\begin{enumerate}
			\item The map $\G\times \X\dashrightarrow \X$, $(g,x) \longmapsto \alpha(g)(x)$ is regular on $U$,
			\item For all $g\in \G$, the open subset $U_g=\{x\in \X, (g,x)\in U\}$ is dense in $\X$ and $\alpha(g)$ is regular on $U_g$.
		\end{enumerate}
		\item The map $\alpha$ is a \emph{regular action} of $\G$ on $\X$ if the map $\G\times \X \rightarrow \X$, $(g,x) \mapsto \alpha(g)(x)$ is a morphism of varieties.
	\end{enumerate}
\end{definition}

If $\G\subset \Bir(\X)$ is an algebraic subgroup and $\phi\colon\X\dashrightarrow \Y$ is a birational map, there exists a unique rational action of $\G$ on $\Y$ which is induced by $\phi$ and such that the following diagram commutes:
\[
\begin{tikzcd}[column sep=small]
\G\times \X \arrow[rr,dashrightarrow] \arrow[d,dashrightarrow,"id\times \phi" left] &&  \X \arrow[d,dashrightarrow,"\phi"]\\
\G\times \Y \arrow[rr,dashrightarrow] && \Y.  
\end{tikzcd}
\] 
A powerful and classical result on rational actions of algebraic groups is the \emph{Regularization Theorem} due to Weil. A modern proof has been given in \cite{Zaitsev} (see also \cite{Kraft}).

\begin{theorem}\cite{Weil}\label{Weil}
	For every rational action of an algebraic group $\G$ on a variety $\X$, there exists a variety $\Y$ and a birational map $\X\dashrightarrow \Y$ such that the induced action of $\G$ on $\Y$ is regular. 
\end{theorem}	

We recall in Lemma \ref{Blanchard} the powerful Blanchard's Lemma.

\begin{definition}\label{equivariant}
	Let $\G$ be an algebraic group acting regularly on varieties $\X$ and $\Y$. A birational map $\phi\colon\X \dashrightarrow \Y$ is \emph{$\G$-equivariant} if the following diagram is commutative:
	\[
	\begin{tikzcd}[column sep=small]
	\G\times \X \arrow[rr] \arrow[d,dashrightarrow,"id\times \phi" left] &&  \X \arrow[d,dashrightarrow,"\phi"]\\
	\G\times \Y \arrow[rr] && \Y.  
	\end{tikzcd}
	\] 
\end{definition}	

\begin{lemma}\label{Blanchard}\cite[Proposition 4.2.1]{BSU}
	Let $\X$ and $\Y$ be varieties and $\phi\colon \X\to \Y$ be a proper morphism such that $\phi_*(\O_\X)=\O_\Y$. Let $\G$ be a connected algebraic group acting regularly on $\X$. Then there exists a unique regular action of $\G$ on $\Y$ such that $\phi$ is $\G$-equivariant.
\end{lemma}

In this text, we will use Blanchard's Lemma in the case where $\X$ and $\Y$ are smooth projective surfaces or smooth curves (and more precisely, $\phi$ will be the contraction of $(-1)$-curves or the structure morphism of a $\PP^1$-bundle). Then $\phi$ induces a morphism of algebraic groups $\phi_* : \Autzero(\X) \to \Autzero(\Y)$. The following proposition is a classical result (see also \cite[Proposition 3.11]{LonjouUrech} for a modern proof using actions on $\operatorname{CAT(0)}$ cubes complexes):

\begin{proposition}\label{minimalsurfaces}
	Let $\X$ be a surface and $\G$ be a connected algebraic subgroup of $\Bir(\X)$. Then $\G$ is conjugate to an algebraic subgroup of $\Autzero(\S)$, where $\S$ is a minimal surface i.e.\ a smooth projective surface without $(-1)$-curves.
\end{proposition}	

\begin{proof}
	First we can apply the Regularization Theorem of Weil on $\X$  to get a surface $\Y$ birationally equivalent to $\X$ and equipped with a regular action of $\G$. Replace $\Y$ by its smooth locus and from \cite[Theorem 1]{Brion}, there exists a non empty open subset $U$ of $\Y$ which is $\G$-stable and quasi-projective. Then from \cite[Theorem 2]{Brion}, the open $U$ admits a $\G$-equivariant completion into a projective variety $\overline{\Y}$ which can be desingularized (\cite{Lipman}). From \cite[Remark B p.155]{Lipman}, there exists a birational morphism $\delta \colon \widetilde{\Y} \to \overline{\Y}$ such that $\widetilde{\Y}$ is a smooth projective variety and $\delta$ is obtained by successive blowups of singular points and normalizations. Hence the action of $\G$ over $\overline{\Y}$ lifts to $\widetilde{\Y}$ so that $\delta$ is $\G$-equivariant. The contraction of $(-1)$-curves of $\widetilde{\Y}$ is $\G$-equivariant from Blanchard's Lemma, so we conclude that $\G$ is conjugate to an algebraic subgroup of $\Autzero(\S)$, where $\S$ is a minimal surface.
\end{proof}	

Apply Proposition \ref{minimalsurfaces} to a surface $\X$ birationally equivalent to $\C\times \PP^1$ with $\C$ a curve. Then from \cite[Examples V.5.8.2, V.5.8.3 and Remark V.5.8.4]{Hartshorne}, the minimal surface $\S$ is either $\PP^2$ or a ruled surface over $\C$. The following lemma will be useful to check if $\Autzero(\S)$ is a maximal connected algebraic subgroup of $\Bir(\S)$.

\begin{lemma}\label{maximal}
	Let $\S$ be a projective surface and $\G$ be a connected algebraic subgroup of $\Aut(\S)$. Then the following hold:
	\begin{enumerate}
		\item The algebraic subgroup $\G$ is maximal if and only if for every projective surface $\T$ and $\G$-equivariant birational map $\phi\colon \S\dashrightarrow \T$, we have $\phi \G \phi^{-1} = \Autzero(\T)$.
		\item For every projective surface $\T$ and $\G$-equivariant birational map $\phi \colon \S \dashrightarrow \T$, there exist $\beta\colon\X \rightarrow \S$ and $\kappa\colon\X\rightarrow \T$ compositions of blowups of fixed points of the $\G$-action such that $\phi=\kappa\beta^{-1}$.
		\item Assume moreover that $\S$ is a minimal surface and $\Autzero(\S)$ acts on $\S$ without fixed points, then every $\Autzero(\S)$-equivariant birational map $\phi\colon \S\dashrightarrow \T$ with $\T$ projective, is an isomorphism. In particular $\Autzero(S)$ is maximal.
	\end{enumerate}
\end{lemma}

\begin{proof} \
	\begin{enumerate}[wide] 
		\item If $\phi\colon\S\dashrightarrow \T$ is a $\G$-equivariant birational map, we have $\phi \G\phi^{-1} \subset \Autzero(\T)$ (see Definition \ref{equivariant}). Since $\G$ is maximal, the inclusion is an equality. Conversely assume by contraposition that $\G$ is not maximal, then it is strictly contained in a connected algebraic subgroup $\H$ of $\Bir(\S)$. From Proposition \ref{minimalsurfaces} there exists a minimal surface $\T$ and a birational map $\phi:\S\dashrightarrow \T$ such that $\H$ is conjugate to a connected algebraic subgroup of $\Autzero(\T)$, i.e.\ $\phi \G \phi^{-1} \varsubsetneq \Autzero(\T)$.
		\item Every birational map $\phi \colon\S \dashrightarrow \T$ can be decomposed as $\phi = \kappa \beta^{-1}$ with $\beta\colon\X \rightarrow \S$ and $\kappa\colon\X\rightarrow \T$ compositions of blowups of smooth points, and we can assume that $\kappa$ and $\beta$ do not contract the same $(-1)$-curves in $\X$. Then for all $g\in \G$:
		\[
		\Bs(\phi)=\Bs(g^{-1}\phi g)=\Bs(\phi g) = g^{-1}(\Bs(\phi)).
		\]
		Therefore, the base points of $\phi$ are fixed points for the $\G$-action, so $\beta$ consists in the blowup of fixed points of the $\G$-action, which is $\G$-equivariant by the universal property of the blowup. Similarly, the morphism $\kappa$ is also $\G$-equivariant.
		\item Because $\S$ is minimal, there is no contraction and since $\Autzero(\S)$ has no fixed point, there is no $\Autzero(\S)$-equivariant blowup. Therefore, every $\Autzero(\S)$-equivariant map from $\S$ to a projective surface $\T$ is an isomorphism from $(2)$. From $(1)$, $\Autzero(\S)$ is maximal .
	\end{enumerate}
\end{proof}

\subsection{Generalities on ruled surfaces}

First we want to classify algebraic subgroups of $\Bir(\C\times \PP^1)$ as stated in Theorem \ref{B}. Then Proposition \ref{minimalsurfaces} suggests studying the maximality of $\Autzero(\S)$ when $\S$ is a minimal surface birationally equivalent to $\C\times \PP^1$, i.e.\ a geometrically ruled surface. Since this object will play an important role, we recall in this subsection the definition and some basic properties.

\begin{definition}\label{defruled}
	A \emph{geometrically ruled surface}, or simply \emph{ruled surface}, is a surface $\S$ equipped with a morphism $\pi \colon\S\to \C$ where $C$ is a smooth curve, and such that all fibers of $\pi$ are isomorphic to $\PP^1$. A \emph{section} of $\S$ is a morphism $\s\colon\C\to \S$ such that $\pi\s=id$. We will also call a \emph{section} the image of $\s$, that is a closed curve $D=\s(\C)$ such that $\pi_{|D}$ is an isomorphism.
\end{definition}

Notice that Definition \ref{defruled} is equivalent to the definition of geometrically ruled surface given in \cite[Section V.2]{Hartshorne}, since Hartshorne mentions that the existence of a section is provided by Tsen's theorem.

\begin{definition} \label{p1bundle} 
	A \emph{$\PP^1$-bundle} $\S$ over a curve $\C$ is a morphism $\pi \colon\S\to \C$ endowed with an open cover $(U_i)_i$ of $\C$ with isomorphisms $g_i\colon\pi^{-1}(U_i)\rightarrow U_i\times \PP^1$, such that for all $i$ the following diagram commutes:
	\[
	\begin{tikzcd}[column sep=small]
	\pi^{-1}(U_i) \arrow[rr,"g_i"] \arrow[dr,swap,"\pi"] &&  U_i\times \PP^1 \arrow[dl,"p_1"]\\
	& U_i, 
	\end{tikzcd}
	\]
	where $p_1$ denotes the projection on the first factor. The morphism $\pi$ is called the \emph{structure morphism} and the open cover $(U_i,g_i)_i$ is called a \emph{trivializing open cover} of $\C$. We denote by $U_{ij}$ the open subset $U_i \cap U_j$ and the \emph{transition maps} are $\tau_{ij}\in \PGL_2(\O_{\C}(U_{ij} ))$ so that $g_i g_j^{-1}$ is equal to:
	\begin{align*}
		U_{ij} \times \PP^1 & \to U_{ij} \times \PP^1 \\
		(x,[u:v]) & \mapsto \left(x,\tau_{ij} (x) \cdot [u:v] \right).
	\end{align*}
	Let $\pi_1 \colon \S_1\to \C$ and $\pi_2 \colon \S_2\to \C$ be $\PP^1$-bundles over $\C$. A \emph{$\C$-isomorphism} (or an isomorphism of $\PP^1$-bundles) $f\colon \S_1\to \S_2$ is an isomorphism of varieties such that $\pi_1 = \pi_2 f$. If moreover $\S_1=\S_2$ then $f$ is called a \emph{$\C$-automorphism} of $\S$. We denote by $\AutC(\S)\subset \Aut(\S)$ the subgroup of $\C$-automorphisms of $\S$.
\end{definition}

From Definition \ref{p1bundle}, we see that a $\PP^1$-bundle over $\C$ is also a ruled surface over $\C$. Conversely, ruled surfaces $\pi\colon \S\to \C$ are also $\PP^1$-bundles over $\C$ (see e.g.\ \cite[Proposition V.2.2]{Hartshorne}). If $\V$ is a vector bundle of rank 2 over $\C$, we denoted by $\PP(\V)$ the $\PP^1$-bundle over $\C$ obtained by projectivization of $\V$ and moreover, all $\PP^1$-bundles over $\C$ are obtained by projectivization of a vector bundle of rank 2 over $\C$ (see e.g.\ \cite[II. Exercise 7.10]{Hartshorne}).

\begin{definition}
	Let $\V$ be a vector bundle of rank $2$. We say that $\V$ is \emph{decomposable} if $\V\simeq \L_1 \oplus \L_2$, for some $\L_1$ and $\L_2$ line subbundles of $\V$. If $\V$ is not decomposable, we say that $\V$ is \emph{indecomposable}. We also say that $\PP(\V)$ is \emph{decomposable} (resp. \emph{indecomposable}) if $\V$ is decomposable (resp. indecomposable).
\end{definition} 

\begin{lemma}\label{automorphismssections}
	Let $\S$ be a $\PP^1$-bundle over a curve $\C$ and let $\s_1$, $\s_2$, $\s_3$ be sections of $\S$. The following hold:
	\begin{enumerate}
		\item There exists a trivialization of $\S$ such that $\s_1$ is the infinite section: i.e.\ for all $U_i$ trivializing open subset of $\C$ we have ${\s_1}_{|U_i}(x) = (x,[1:0])$ and the transition maps of $\S$ are upper triangular matrices:
		\begin{align}
		U_{ij} & \rightarrow  \PGL_2(\O_\C(U_{ij}))\nonumber\\ 
		x & \mapsto
		\begin{bmatrix}
		a_{ij}(x) & c_{ij}(x) \\
		0 & b_{ij}(x)
		\end{bmatrix}. \nonumber
		\end{align}
		\item If $\s_1$ and $\s_2$ are disjoint then there exists a trivialization of $\S$ such that $\s_1$ is the infinite section and $\s_2$ is the zero section, i.e.\ for all $U_i$ trivializing open subset of $\C$ we have ${\s_2}_{|U_i}(x) = (x,[0:1])$. Moreover, the transition maps of $\S$ are diagonal matrices:
		\begin{align}
		U_{ij} & \rightarrow  \PGL_2(\O_\C(U_{ij}))\nonumber\\
		x & \mapsto
		\begin{bmatrix}
		a_{ij}(x) & 0 \\
		0 & b_{ij}(x)
		\end{bmatrix}. \nonumber
		\end{align}
		\item If $\s_1$, $\s_2$, $\s_3$ are pairwise disjoint then there exists a trivialization of $\S$ such that: $\s_1$ is the infinite section, $\s_2$ is the zero section, and $\s_3$ is the section defined on all $U_i$ trivializing open subset of $\C$ as ${\s_3}_{|U_i}(x) = (x,[1:1])$. Moreover, $\S$ is isomorphic to $\C\times \PP^1$.
	\end{enumerate}
\end{lemma}

\begin{proof}\
	\begin{enumerate}[wide] 
		\item 
		Let $(U_i)_i$ be a trivializing open cover of $\C$. For all $i$, we have ${\s_1}_{|U_i}(x)= (x,[u_{1i}(x):v_{1i}(x)])$ with $u_{1i},v_{1i}\in \O_\C(U_i)$. If $u_{1i}$ and $v_{1i}$ both vanish at $z\in U_i$ with respectively multiplicities $m_u$ and $m_v$, $m=\min (m_u,m_v)$, and $f$ is a local parameter at $z$, then ${\s_1}_{U_i}(x)=(x,[u_{1i}(x)/f(x)^m:v_{1i}(x)/f(x)^m])$. We can assume that $u_{1i}$ and $v_{1i}$ do not vanish simultaneously and by refining the open cover $(U_i)_i$, we can also assume that either $u_{1i}\in {\O_\C}^*(U_i)$ or $v_{1i}\in {\O_\C}^*(U_i)$. Then one can compose ${\s_1}_{|U_i}$ and the charts on the left by the automorphisms of $U_i\times \PP^1$:	
		\[
		(x,[u:v]) \mapsto  \left\{
		\begin{array}{ll}
		\left(x,
		\begin{bmatrix}
		1 & 0 \\
		-v_{1i}(x) & u_{1i}(x)
		\end{bmatrix}
		\cdot \begin{bmatrix}
		u\\v
		\end{bmatrix}
		\right)
		\mbox{if } u_{1i}(x)\neq 0\mbox{ on } U_i, \vspace{0.05cm}\\
		\left(x,
		\begin{bmatrix}
		0 & 1 \\
		-v_{1i}(x) & u_{1i}(x)
		\end{bmatrix}
		\cdot \begin{bmatrix}
		u\\v
		\end{bmatrix}
		\right)
		\mbox{if } v_{1i}(x)\neq 0\mbox{ on } U_i.
		\end{array}
		\right.
		\]
		Under this trivialization of $\S$, the section $\s_1$ is the infinite section and $[1:0]$ is preserved by the transition maps, which have to be upper triangular matrices.
		\item First apply (1) so that $\s_1$ is the infinite section. If ${\s_2}_{|U_i}(x)= (x,[u_{2i}(x):v_{2i}(x)])$, we also can assume $v_{2i}\in \O^*_\C(U_i)$ as in $(1)$ if needed. Then we compose by the following automorphisms of $U_i\times \PP^1$:
		\[
		(x,[u:v]) \mapsto  
		\left(x,
		\begin{bmatrix}
		v_{2i}(x) & -u_{2i}(x) \\
		0 & 1
		\end{bmatrix}
		\cdot \begin{bmatrix}
		u\\v
		\end{bmatrix}
		\right).
		\]
		Under this trivialization of $\S$, the section $\s_1$ remains the infinite section and $\s_2$ is the zero section. Moreover, $[1:0]$ and $[0:1]$ are preserved by the transition maps, which have to be diagonal matrices.
		\item First apply (2) so that $\s_1$ is the infinite section and $\s_2$ is the zero section. On the trivializing open subset $U_i$, we can write $\s_3(x)= (x,[u_{3i}(x):v_{3i}(x)])$ with $u_{3i}, v_{3i}\in \O^*_\C(U_i)$ as in $(1)$. Then we compose by the following automorphisms of $U_i\times \PP^1$:
		\[
		(x,[u:v]) \mapsto  
		\left(x,
		\begin{bmatrix}
		1/u_{3i}(x) & 0 \\
		0 & 1/v_{3i}(x)
		\end{bmatrix}
		\cdot \begin{bmatrix}
		u\\v
		\end{bmatrix}
		\right).
		\]
		Under this trivialization of $\S$, the sections $\s_1$ and $\s_2$ remain respectively the infinite section and the zero section; and $\s_3$ is the constant section $x\mapsto (x,[1:1])$ on every trivializing open subset of $\C$. It implies that the transition maps of $\S$ are the identity matrices i.e.\ $\S$ is trivial.
	\end{enumerate}
\end{proof}

\begin{lemma}\label{sub<->section}
	Let $\PP(\V)$ be a $\PP^1$-bundle over a curve $\C$, and $\s$ be a section of $\PP(\V)$ given locally by:
	\begin{align*}
		\s_i \colon U_i & \to U_i \times \PP^1 \\
		x & \mapsto (x,[u_i(x):v_i(x)]).
	\end{align*}
	For all $i$, we define $\L_i = \{ \left(x,(\lambda u_i(x),\lambda v_i(x) )\right) \in U_i \times \AA^2, \lambda \in \kk\}\simeq U_i \times \AA^1$ and the line subbundle $\pi\colon \L(\s) \to \C$ of $\V$ such that $\pi^{-1} (U_i) = \L_i$. Then the following hold:
	\begin{enumerate}
		\item The map $\s \mapsto \L(\s)$ is a bijection between the set of sections of $\PP(\V)$ and the set of line subbundles of $\V$ over $\C$.
		\item Two sections $\s_1$ and $\s_2$ are disjoint if and only if $\PP(\V)$ is $\C$-isomorphic to $\PP(\L(\s_1)\oplus \L(\s_2))$.
	\end{enumerate} 
\end{lemma}

\begin{proof}\
	\begin{enumerate}[wide]
		\item 
		 If $\L$ is a line subbundle of $\V$, we have for all $i$ an embedding $\L_{|U_i}\hookrightarrow \V_{|U_i}$ which induces by projectivisation an embedding $\s_i\colon U_i \to U_i \times \PP^1$. Since the family $(\L_{|U_i})_i$ glues into $\L$, the morphism $(\s_i)_i$ glue into a section $\s$ of $\PP(\V)$. This construction is the inverse of the map $\s\mapsto \L(\s)$.
		\item Let $\s_1$ and $\s_2$ be disjoint sections of $\PP(\V)$. From Lemma \ref{automorphismssections} (2), we can assume that $\s_1$ is the infinite section and $\s_2$ is the zero section, and the transition maps of $\PP(\V)$ are:
		\begin{align}
		U_{ij} & \rightarrow  \PGL_2(\O_\C(U_{ij}))\nonumber\\
		x & \mapsto
		\begin{bmatrix}
		a_{ij}(x) & 0 \\
		0 & b_{ij}(x)
		\end{bmatrix}. \nonumber
		\end{align}
		The coefficients $a_{ij}$ and $b_{ij}$ don't vanish on $U_{ij}$, and we can choose $x\mapsto a_{ij}(x)$ as the transition maps of $\L(\s_1)$ and $x\mapsto b_{ij}(x)$ as the transition maps of $\L(\s_2)$, i.e.\ $\PP(\V)\simeq \PP(\L(\s_1)\oplus \L(\s_2))$. Conversely if we have $\PP(\V)\simeq \PP(\L(\s_1)\oplus \L(\s_2))$, one can choose a trivializing open cover so that the transition maps are given by diagonal matrices. Under this choice, the section $\s_1$ is the zero section and $\s_2$ is the infinite section, thus they are disjoint.
	\end{enumerate}
\end{proof}

\subsection{Segre invariant}

In this subsection we recall the Segre invariant and its properties. This invariant has already been used by Maruyama in his classification of ruled surfaces \cite{Maruyama2}. One can also check that the Segre invariant corresponds to $-e$, where $e$ is the invariant defined in \cite[V. Proposition 2.8]{Hartshorne}. If $c$ and $c'$ are curves in a smooth projective surface $\S$, we denote by $c\cdot c'$ their intersection number and if $c=c'$, we denote it by $c^2$.

\begin{definition}\label{defsegre}
	Let $\S\to \C$ be a ruled surface. The \emph{Segre invariant} $\seg(\S)$ of $\S$ is defined as the quantity:
	\[
	\min \{\s^2, \s\text{ section of } \S\}.
	\]
	A \emph{minimal section} of $\S$ is a section $\s$ of $\S$ such that $\s^2 = \seg(\S)$.
\end{definition}

The Segre invariant is well-defined since all ruled surfaces are obtained from $\C\times \PP^1$ by finitely many elementary transformations and $\seg(\C\times \PP^1)=0$ (see Lemma \ref{sectiontrivialbundle}).

\begin{definition}
	A line subbundle $\M$ of a vector bundle $\V$ is \emph{maximal} if its degree is maximal among all line subbundles of $\V$.
\end{definition}

One can use Riemann-Roch theorem to show that the degree of line subbundles is bounded above, but it follows also from Proposition \ref{selfintersectiondegree} and the fact that the Segre invariant is well-defined.

In explicit computations, we will often use that the group of divisors up to numerical equivalence of a ruled surface $S$ is generated by the class of a section $\sigma$ and a fibre $f$, and they satisfy $f^2=0$ and $\sigma\cdot f =1$ (see \cite[Proposition V.2.3]{Hartshorne}). The next lemma is partially contained in \cite[Exercise V.2.4]{Hartshorne}. 
\begin{lemma}\label{sectiontrivialbundle}
	Let $\C$ be a curve of genus $\g$ and $\s$ be a section of $\C\times \PP^1$ defined as $\C \overset{\s}{\rightarrow} \C\times \PP^1$, $x \mapsto (x,g_\s(x))$ where $g_\s:\C\rightarrow \PP^1$ is a morphism. Then:
	\[
	\s^2 = 2\deg(g_\s),
	\]
	and in particular, each section of $\C\times \PP^1$ has an even and non negative self-intersection number. In particular, $\seg(\C\times \PP^1)=0$. Moreover, if $\g>0$ then there is no section of self-intersection $2$, and if $\g=1$ then there exist sections of self-intersection $4$.
\end{lemma}

\begin{proof}
	Sections of $\C\times \PP^1$ are of the form $\C \overset{\s}{\rightarrow} \C\times \PP^1$, $x \mapsto (x,g_\s(x))$ where $g_\s:\C\rightarrow \PP^1$ is a morphism. Moreover, the group of $\operatorname{Num}(\C\times \PP^1)$ is generated by the class of a constant section $\s_c$ and the class of a fiber $f$ (\cite[Proposition V.2.3]{Hartshorne}). Therefore, any section $\s$ of $\C\times \PP^1$ is numerically equivalent to $a\s_c + bf$ for some integers $a$, $b$. Intersecting $\s$ with $f$ and with $\s_c$, one finds respectively that $a=1$ and $b=\s_c\cdot \s$. Since all constant sections are linearly equivalent, and for a general constant section the quantity $\s_c\cdot \s$ corresponds to $\deg(g_\s)$, we get that $\s\equiv \s_c + \deg(g_\s)f$. Consequently, we have $\s^2=2\deg(g_\s)\geq 0$. Because constant sections have self-intersection $0$, it follows that $\seg(\C\times \PP^1)=0$. If moreover $\g>0$ then there does not exist a morphism $\C\rightarrow \PP^1$ of degree $1$ and it implies that there is no section of self-intersection $2$ in $\C\times \PP^1$. But if $\g=1$ then there exist morphisms $\C \to \PP^1$ of degree $2$, hence there exist sections of self-intersection $4$ in $\C\times \PP^1$.
\end{proof}

  In \cite[Lemma 1.15]{Maruyama2} has been stated Corollary \ref{Segre} which provides an alternative way to compute the Segre invariant of a ruled surface. However, it is a consequence of the more general and useful statement given in Proposition \ref{selfintersectiondegree}, which also follows from Maruyama's proof. We give a simple proof of Proposition \ref{selfintersectiondegree} based on a direct computation in local charts.

\begin{proposition}\label{selfintersectiondegree}
	Let $\pi \colon\PP(\V)\rightarrow \C$ be a $\PP^1$-bundle and $\s$ be a section. Then the following equality holds:
	\[
	\s^2 = \deg(\V) - 2\deg(\L(\s)),
	\]
	where $\deg(\V)$ is the degree of the determinant line bundle of $\V$.
\end{proposition}

\begin{proof}
	From Lemma \ref{automorphismssections} (1), we can assume that $\s$ is the infinite section and the transition maps of $\PP(\V)$ are upper triangular matrices. Let $(U_0,g_0)$ be a trivializing open subset of $\V$, and let $\s_0$ and $f_0$ be defined as below:
	\begin{align*}
		\s_0\colon U_0 & \longrightarrow U_0 \times \PP^1 & f_0\colon U_0  \times \PP^1 & \dashrightarrow \kk \\
		x & \longmapsto (x,[0:1]) ,&  (x,[u:v])  & \longmapsto \frac{u}{v}.
	\end{align*}
	 The morphism $\s_0$ extends to a section defined over $\C$ which is disjoint from $\s$ on $U_0$, and $f_0$ extends to a rational function $f$ over $\PP(\V)$. We have:
	\[
	 \div(f)  = \s_0 - \s + \sum_{z \in \C\setminus U_0} \nu_{\pi^{-1}(z)}(f)\cdot \pi^{-1}(z),
	\]
	where $\nu_{\pi^{-1}(z)}$ denotes the valuation of the fiber $\pi^{-1}(z)$. In consequence:
	\begin{equation}
	\s^2 = \s\cdot \s_0 + \sum_{z \in \C\setminus U_0} \nu_{\pi^{-1}(z)}(f). \label{eqs^2}
	\end{equation}
    The subset $\C\setminus U_0$ has finitely many points and for each $z\in \C\setminus U_0$, we can choose a trivializing open neighborhood $(U_{z},g_{z})$ of $z$ and we denote by $\tau_{0z}$ the transition map defining $g_0g_z^{-1}$:
    \begin{align*}
    \tau_{0z}\colon U_{0z} & \rightarrow  \PGL_2(\O_\C(U_{0z}))\\
    x & \mapsto
    \begin{bmatrix}
    a_{0z}(x) & c_{0z}(x) \\
    0 & b_{0z}(x)
    \end{bmatrix}.
    \end{align*}
    The coordinates of the section $\s_0$ above $U_{0z}\subset U_z$ are solutions of the equation $\tau_{0z}(x)\cdot [u(x):v(x)] = [0:1]$, i.e.\ $g_zg_0^{-1}\sigma_0(x) = (x,[-c_{0z}(x):a_{0z}(x)])$. The rational map $x\mapsto [-c_{0z}(x):a_{0z}(x)]$ extends to $U_z$ and if we denote by $\nu_z$ the valuation in $\O_{\C,z}$, then the two sections $\s$ and $\s_0$ intersect above $z$ if and only if $\nu_z(c_{0z}) < \nu_z(a_{0z})$. When they intersect, the intersection number equals $\nu_z(a_{0z}) - \nu_z(c_{0z})$: this quantity is independent of the choice of the trivializing open subset $U_z$ and of the choice of $a_{0z},b_{0z},c_{0z}$. It implies that: 
    \begin{equation}
     	\s\cdot \s_0= \sum_{z \in \C\setminus U_0} \max(\nu_z(a_{0z})-\nu_z(c_{0z}),0). \label{ss0}
    \end{equation}
   Moreover, we have $f_{|U_z} (x,[u:v])= f_0 \left(x,\tau_{0z}(x)\cdot[u:v]\right) = \frac{a_{0z}(x)u + c_{0z}(x)v}{b_{0z}(x)v}$ since the following diagram is commutative:
   \[
   \begin{tikzcd}[column sep=small]
   {U_{0z}}\times \PP^1 \arrow[rr,"g_0 g_z^{-1}"] \arrow[dr,dashed,"f_{|U_z}" swap] && {U_{0z}}\times \PP^1 \arrow[dl,dashed,"f_0"]  \\ 
   &  \kk.
   \end{tikzcd}
   \]
   Let $r=\min( \nu_z(a_{0z}) , \nu_z(c_{0z}) )$ and $s=\nu_z(b_{0z})$. If $t$ is a generator of the maximal ideal $\mathfrak{m}_{\C,z}\subset \O_{\C,z}$ then there exist $\widetilde{a}_{0z},\widetilde{b}_{0z},\widetilde{c}_{0z}\in \O_{\C}(U_{0z})$ with $\widetilde{b}_{0z} \in \O_{\C,z}^*$ and $\widetilde{a}_{0z}$ or $\widetilde{c}_{0z}$ in $\O_{\C,z}^*$, such that $a_{0z} = t^r \widetilde{a}_{0z}$, $c_{0z} = t^r \widetilde{c}_{0z}$ and $b_{0z}(x) = t^s \widetilde{b}_{0z}$. We obtain:
   \[
   \nu_{\pi^{-1}(z)}(f)  =  \nu_{\pi^{-1}(z)} \left( \frac{a_{0z}(x)u + c_{0z}(x)v}{b_{0z}(x)v} \right)  = (r-s)+\nu_{\pi^{-1}(z)}\left(\frac{\widetilde{a}_{0z}(x)u+\widetilde{c}_{0z}(x)v}{\widetilde{b}_{0z}(x)v} \right), 
   \] 
   and $\nu_p \left( \frac{\widetilde{a}_{0z}(x)u+\widetilde{c}_{0z}(x)v}{\widetilde{b}_{0z}(x)v}\right)=0$ for a general point $p\in \pi^{-1}(z)$. Therefore:
   \begin{equation}
		   \nu_{\pi^{-1}(z)}(f) = r-s=\min(\nu_z(a_{0z}),\nu_z(c_{0z}) ) - \nu_z(b_{0z} ), \label{valuationfiber}
   \end{equation}
	which is also independent of the choice of the trivializing open $U_z$ and of the choice of $a_{0z},b_{0z},c_{0z}$. Then by substituting (\ref{ss0}) and (\ref{valuationfiber}) in (\ref{eqs^2}), we get:
	\[
	\s^2  = \sum_{z \in \C\setminus U_0} \nu_z(a_{0z})-\nu_z(b_{0z}).
	\]
	Since $\s$ is the infinite section, the line subbundle $\L(\s)$ of $\V$ is defined by $\{(x,(\lambda,0))\in U\times \AA^2, \lambda \in \kk\}$ on every trivializing open subset $U$. Hence we can choose the transition map of $\L(\s)$ on $U_{0z}$ as $x \mapsto a_{0z}(x)$ and the transition maps of $\V/\L(\s)$ on $U_{0z}$ as $x\mapsto b_{0z}(x)$. Let $a\colon \C \to \L(\s) \subset \V$ and $b\colon \C \to \V/\L(\s)$ be the rational sections defined by: 
	\begin{align*}
		a\colon U_0 & \to \L(\s)_{|U_0} &b\colon U_0 & \to \left(\V/\L(\s)\right)_{|U_0} \\
		x & \mapsto (x,1) , & x&\mapsto (x,1).
	\end{align*}
	Up to a multiple, we have that $a_{0z}^{-1}$ and $b_{0z}^{-1}$ are respectively the coordinates of the sections $a$ and $b$ on $U_z$. Finally we have that $\s^2 =\sum_{z \in \C} \nu_z(b)-\nu_z(a)=\deg\left(\V/\L(\s)\right)-\deg(\L(\s))  $. 
	Using the additivity of the degree on the short exact sequence $0\to \L(\sigma) \to \V \to \V/\L(\sigma) \to 0$, we deduce that 
	$$\sigma^2= \deg(\V) - 2\deg(\L(\s)).$$
\end{proof}
	
Propositions \ref{sub<->section} $(1)$ and \ref{selfintersectiondegree} imply that $\s$ is a minimal section of $\PP(\V)$ if and only if $\L(\s)$ is a maximal line subbundle of $\V$. In particular, we have the following corollary that can also be found in \cite[Proposition V.2.9]{Hartshorne}:

\begin{corollary}\label{Segre}
	Let $\S=\PP(\V)$ be a $\PP^1$-bundle over a curve $\C$ and $\M$ be a maximal line subbundle of $\V$. Then the following equality holds:
	\[
	\seg(\S) = \deg(\V) - 2\deg(\M),
	\]
	where $\deg(\V)$ is the degree of the determinant line bundle of $\V$.
\end{corollary}

The main use of the Segre invariant is given in Proposition \ref{numberminimalsection}, which is partially stated in \cite[Corollary 1.17]{Maruyama2} without proof, and partially proven in \cite[Theorem V.2.12]{Hartshorne}.

\begin{lemma}\label{sectionnumeq}
	Let $\S$ be a $\PP^1$-bundle over a curve $\C$. Two sections of $\S$ having the same self-intersection number are numerically equivalent.
\end{lemma}

\begin{proof}
	Let $\s_1$ and $\s_2$ be sections of $\S$ having the same self-intersection number. Let $\operatorname{Num}(\S)$ be the group of divisors classes up to numerical equivalence. Then $\operatorname{Num}(\S)\otimes_\mathbb{Z} \mathbb{Q}$ is generated by $\mathrm{K}_\S$ and a fiber $f$ (see \cite[Proposition V.2.3]{Hartshorne}), and from the arithmetic genus formula: 
	\[
	\frac{1}{2}(\mathrm{K}_\S + \s_1 )\cdot \s_1 +1 = g(\C) = \frac{1}{2}(\mathrm{K}_\S + \s_2 )\cdot \s_2 +1, 
	\]
	it implies that $\mathrm{K}_\S\cdot \s_1 = \mathrm{K}_\S \cdot \s_2$. In particular, the sections $\s_1$ and $\s_2$ are numerically equivalent.
\end{proof}

\begin{proposition}\label{numberminimalsection}
	Let $\S=\PP(\V)$ be a $\PP^1$-bundle over a curve $\C$. The following assertions hold:
	\begin{enumerate}
		\item if $\seg(\S)>0$ then $\S$ is indecomposable.
		\item if $\seg(\S)<0$ then $\S$ admits a unique minimal section. 
		\item if $\seg(\S)=0$ then: 
		\begin{enumerate}
			\item any two distinct minimal sections of $\S$ are disjoint, 
			\item $\S$ is indecomposable if and only if $\S$ has a unique minimal section, 
			\item $\S$ is decomposable and not trivial if and only if $\S$ has exactly two minimal sections,
			\item $\S$ is trivial if and only if $\S$ has at least three minimal sections.
		\end{enumerate}
	\end{enumerate}
\end{proposition}

\begin{proof}\
	\begin{enumerate}[wide]
		\item  Assume that $\V\simeq \L_1\oplus \L_2$ is decomposable and let $\M$ be a maximal line subbundle of $\V$. Then $\deg(\V)=\deg(\L_1)+\deg(\L_2)\leq 2\deg(\M)$ and from Corollary $\ref{Segre}$: $\seg(\S)=\deg(\V)-2\deg(\M )> 0$ which is a contradiction.
		\item We assume that $\S$ admits two distinct minimal sections $\s_1$ and $\s_2$. From Lemma \ref{sectionnumeq}, the sections $\s_1$ and $\s_2$ are numerically equivalent and therefore $\seg(\S) = \s_2\cdot \s_1 \geq 0$ and it is a contradiction.
		\item We assume $\seg(\S)=0$:
		\begin{enumerate}[wide]
			\item Let $\s_1$ and $\s_2$ be distinct minimal sections. Since they are numerically equivalent from Lemma \ref{sectionnumeq}: $0=\seg(\S)= \s_1^2 = \s_1 \cdot \s_2$ i.e.\ $\s_1$ and $\s_2$ are disjoint sections.
			\item Assume by contraposition that $\S$ has two minimal sections: since they are disjoint from (i), it implies that $\S$ is decomposable. To prove the converse we assume by contraposition that $\S$ is decomposable, i.e.\ $\C$-isomorphic to $\PP(\L(\s_1)\oplus \L(\s_2))$ for some sections $\s_1$ and $\s_2$, and then we have $0=\seg(\S)=\deg(\L(\s_1))+\deg(\L(\s_2))-2\deg(\M)$. It implies that $\L(\s_1)$ and $\L(\s_2)$ are maximal line subbundles i.e.\ $\s_1$ and $\s_2$ are disjoint minimal sections of $\S$. By Lemma \ref{sub<->section}, $\S$ is decomposable.
			\item If $\S$ is decomposable then $\S$ admits at least two minimal sections from (ii). Assume that $\S$ has three distinct minimal sections then they are pairwise disjoints from (i) and it implies that $\S$ is trivial from Lemma \ref{automorphismssections} (3). So if $\S$ is decomposable and non trivial then $\S$ has exactly two minimal sections. Conversely, assume that $\S$ has exactly two minimal sections $\s_1$ and $\s_2$: it follows again from (i) and Lemma \ref{sub<->section} (2) that $\S$ is $\C$-isomorphic to $\PP(\L(\s_1)\oplus \L(\s_2))$ which is decomposable. Since the trivial bundle has infinitely many minimal sections, $\S$ cannot be the trivial bundle.
			\item It follows from the equivalences of (ii) and (iii).
		\end{enumerate}
	\end{enumerate}		
\end{proof}

\subsection{Construction of some ruled surfaces with elementary transformations}\label{ruledelemtrans}

Let $\C$ be a curve and $z_1,z_2\in \C$. We denote by $\S_{z_1,z_2}$ the ruled surface $\PP(\O_\C(z_1)\oplus \O_\C(z_2))$. If $\C$ is an elliptic curve, we denote by $\A_{0,\C}$ and $\A_{1,\C}$ the Atiyah $\PP^1$-bundles defined in Theorem \ref{Atiyah}. The classification of vector bundles over an elliptic curve given in \cite{Atiyah} is much more general than the statement we need for ruled surfaces, and the reader can also find it in \cite[Example V.2.11.6 and Theorem V.2.15]{Hartshorne}:

\begin{theorem}\cite[Theorems 5, 6, 7, 10 and 11]{Atiyah} \label{Atiyah}
	Let $\C$ be an elliptic curve and $p\in\C$. There exist two indecomposable vector bundles $\mathcal{F}_{0,\C}$ and $\mathcal{F}_{1,\C}$ of rank $2$ and respectively of degree $0$ and $1$, which fit into the exact sequences:
	\begin{align*}
		0 & \to \O_\C \to \mathcal{F}_{0,\C} \to \O_\C \to 0, \\
		0 & \to \O_\C  \to \mathcal{F}_{1,\C} \to \O_\C(p) \to 0.
	\end{align*}
	Moreover, the isomorphism class of $\PP(\mathcal{F}_{1,\C})$ does not depend on the choice of $p$, and the Atiyah $\PP^1$-bundles $\A_{0,\C}:=\PP(\mathcal{F}_{0,\C})$ and $\A_{1,\C}:=\PP(\mathcal{F}_{1,\C})$ are exactly the two $\C$-isomorphism classes of indecomposable $\PP^1$-bundles over $\C$. 
\end{theorem}

A ruled surface $\S\to \C$ is birationally equivalent to $\C\times \PP^1$. In this subsection, we give explicitly birational maps $\C\times \PP^1\dashrightarrow \S$ when $\S$ is isomorphic to $\A_{0,\C}$, $\A_{1,\C}$ or $\S_{z_1,z_2}$ for some $z_1,z_2\in \C$. We also deduce the Segre invariant of the Atiyah $\PP^1$-bundles and of $\S_{z_1,z_2}$ for all $z_1,z_2\in \C$. Let $\S$ be a $\PP^1$-bundle over a curve $\C$, let $p\in \S$ and $f_p$ the fiber containing $p$. We denote by $\beta_p:\bup_p(\S)\rightarrow \S$ the blowup of $\S$ at $p$ and by $\E_p$ the exceptional divisor. The strict transform of $f_p$ under the birational map $\beta_{p}^{-1}$ is a $(-1)$-curve and we denote by $\kappa_p : \bup_p(\S) \to \T$ its contraction. The \emph{elementary transformation} of $\S$ centered on $p$ is the birational map $\epsilon_p=\kappa_p \beta_p^{-1}$. If $p_1,p_2\in \C\times \PP^1$ are distinct points such that $p_1$ and $p_2$ are not on the same fiber, we denote by $\epsilon_{p_2,p_1}$ the blowups of $p_1$ and $p_2$ followed by the contractions of their respective fibers. 

\begin{proposition}\label{decomposable}
	Let $\C$ be an elliptic curve and $\pi \colon\C\times \PP^1 \to \C$ be the projection on the first factor. Let $p_1,p_2\in \C\times \PP^1$, we denote by $z_1=\pi(p_1)$ and $z_2=\pi(p_2)$. The following hold:
	\begin{enumerate}
		\item The surface $\epsilon_{p_1}(\C\times \PP^1)$ is isomorphic to $\PP(\O_\C(z_1)\oplus \O_\C)$ and $\seg(\PP(\O_\C(z_1)\oplus \O_\C))=-1$. Moreover, the base point $q_1$ of $\epsilon_{p_1}^{-1}$ is the unique point where all the sections of self-intersection $1$ meet.
		\item Assume moreover that $p_1$ and $p_2$ are not in the same fiber and not in the same constant section. Then the surface $\epsilon_{p_2,p_1} (\C\times \PP^1)$ is isomorphic to $\PP(\O_\C(z_1)\oplus \O_\C(z_2))$ and it has exactly two disjoint sections of self-intersection $0$. We have $\seg(\PP(\O_\C(z_1)\oplus \O_\C(z_2)))=0$ and if $q_1,q_2$ are the base points of $\epsilon_{p_2,p_1}^{-1}$, then every section of self-intersection $2$ passing through $q_1$ also passes through $q_2$.
	\end{enumerate}
\end{proposition}

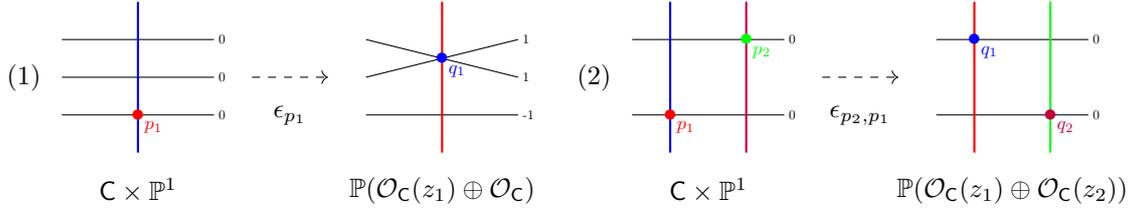
\begin{center}
	\begin{tikzpicture}[scale=0.5]
	\draw (-1,0) node[scale=1]{$\text{(1)}$};
	\draw (0,1) -- (4,1); \draw (4,1) node[scale=0.5,right]{0};
	\draw (0,0) -- (4,0); \draw (4,0) node[scale=0.5,right]{0};
	\draw (0,-1) -- (4,-1); \draw (4,-1) node[scale=0.5,right]{0};
	\draw[thick, blue] (2,-2) -- (2,2) ;
	\draw[red] (2,-1) node[below right,scale=0.7]{$p_1$} node{$\bullet$};
	\draw (2,-3) node[scale=1]{$\C\times \PP^1$};

	\draw [dashed,->] (5,0) -- (7,0);
	\draw (6,-1) node[scale=1]{$\epsilon_{p_1}$};

	\draw (8,-1) -- (12,-1);	\draw (12,-1) node[scale=0.5,right]{-1}; 
	\draw[thick, red] (10,-2) -- (10,2);
	\draw (8,0) -- (12,1); \draw (12,1) node[scale=0.5,right]{1};
	\draw (8,1) -- (12,0); \draw (12,0) node[scale=0.5,right]{1};
	\draw[blue] (10,0.5) node[below right,scale=0.7]{$q_1$} node{$\bullet$};
	\draw (10,-3) node[scale=1]{$\PP(\O_\C(z_1)\oplus \O_\C)$};

	\draw (14,0) node[scale=1]{$\text{(2)}$};	
	\draw (15,-1) -- (19,-1); \draw (19,-1) node[scale=0.5,right]{0};
	\draw (15,1) -- (19,1); \draw (19,1) node[scale=0.5,right]{0};
	\draw[thick, purple] (18,-2) -- (18,2);
	\draw[thick, blue] (16,-2) -- (16,2); 
	\draw[red] (16,-1) node[below right,scale=0.7]{$p_1$} node{$\bullet$};
	\draw[green] (18,1) node[below right,scale=0.7]{$p_2$} node{$\bullet$};
	\draw (17,-3) node[scale=1]{$\C\times \PP^1$};

	\draw [dashed,->] (20,0) -- (22,0);
	\draw (21,-1) node[scale=1]{$\epsilon_{p_2,p_1}$};
	
	\draw (23,-1) -- (27,-1); \draw (27,-1) node[scale=0.5,right]{0};
	\draw (23,1) -- (27,1); \draw (27,1) node[scale=0.5,right]{0};
	\draw[thick, red] (24,-2) -- (24,2);
	\draw[thick, green] (26,-2) -- (26,2); 
	\draw[blue] (24,1) node[below right,scale=0.7]{$q_1$} node{$\bullet$};
	\draw[purple] (26,-1) node[below right,scale=0.7]{$q_2$} node{$\bullet$};
	\draw (25,-3) node[scale=1]{$\PP(\O_\C(z_1)\oplus \O_\C(z_ 2))$};
	\end{tikzpicture}
\end{center}	

\begin{proof}\
	\begin{enumerate}[wide]
		\item Up to a $\C$-automorphism of $\C\times \PP^1$ we can assume that $p_1=(z_1,[1:0])$ and let $U_1$ be an open neighborhood of $z_1$. Let $f\in \kk(\C)^*$ which has a zero of order $1$ at $z_1$, we can also assume that $U_1$ does not contain any zeros and poles of $f$ except at $z_1$. Let $U_0=\C\setminus z_1$, we define:
		\begin{align*}
			\phi_0:U_0\times \PP^1 & \longrightarrow U_0\times \PP^1 & \phi_1 :U_1 \times \PP^1 & \dashrightarrow U_1 \times \PP^1 \\
			(x,[u:v]) & \longrightarrow (x,[u:v]) & (x,[u:v]) & \longrightarrow (x,[f(x)u:v]).
		\end{align*}
		The domains of $\phi_0,\phi_1$ glue into $\C\times \PP^1$ and the codomains glue into a $\PP^1$-bundle over $\C$ through the transition map:
		\begin{align*}
			U_0 \cap U_1 & \rightarrow  \PGL_2(\O_\C(U_0\cap U_1))\nonumber\\
			x & \mapsto
			\begin{bmatrix}
				f(x) & 0 \\
				0 & 1
			\end{bmatrix}.
		\end{align*}
		So $\phi_0$ and $\phi_1$ glue onto a birational map $\phi\colon \C\times \PP^1 \dashrightarrow \PP(\O_\C(z_1)\oplus \O_\C)$. Since $\phi$ has $p_1$ as unique base point of order $1$, one can check by describing the blowups in local charts that $\phi$ is the elementary transformation $\epsilon_{p_1}$. Moreover, the strict transform by $\epsilon_{p_1}$ of the infinite section is the unique section of self-intersection number $-1$ and all the other sections in $\PP(\O_\C(z_1)\oplus \O_\C)$ have self-intersection number at least $1$, so $\seg(\PP(\O_\C(z_1)\oplus \O_\C))=-1$. Because $\C$ is an elliptic curve, there is no section of self-intersection $2$ in $\C\times \PP^1$ from Lemma \ref{sectiontrivialbundle} and the sections of self-intersection $1$ in $\PP(\O_\C(z_1)\oplus \O_\C)$ are exactly the strict transform by $\epsilon_{p_1}$ of the constant sections of $\C\times \PP^1$, so they all pass through $q_1$ and it is their unique common intersection.
		
		\item Similarly we can assume $p_1=(z_1,[1:0])$ and $p_2=(z_2,[0:1])$. Let $U_1$ and $U_2$ be open trivializing neighborhoods of $z_1$ and $z_2$. Let $f,g\in \kk(\C)^*$ having a zero of order one respectively at $z_1$ and at $z_2$. We can assume that $U_1$ does not contain $z_2$ and any zeros or poles of $f$ except $z_1$, and similarly we can assume that $U_2$ does not contain $z_1$ and any zeros or poles of $g$ except $z_2$. Let $U_0=\C\setminus \{z_1,z_2\}$ then we take $\phi_0$ and $\phi_1$ as in (1) and we define:
		\begin{align*}
			\phi_2:U_2\times \PP^1 & \dashrightarrow U_2\times \PP^1 \\
			(x,[u:v])  & \longmapsto (x,[u: g(x)v]).
		\end{align*}
		The maps $\phi_0,\phi_1,\phi_2$ glue into a birational map $\phi : \C\times \PP^1 \dashrightarrow \PP(\O_\C(z_1)\oplus \O_\C(z_2))$. Since $\phi$ has exactly two base points $p_1$ and $p_2$ of order $1$, one can check by local equations of blowups that $\phi$ equals $\epsilon_{p_2,p_1}$. From Lemma \ref{sectiontrivialbundle} there is no section of self-intersection $2$ in $\C\times \PP^1$, so the strict transform by $\epsilon_{p_2,p_1}$ of the infinite section and of the zero section of $\C\times \PP^1$ are the only sections of self-intersection number $0$ in $\PP(\O_\C(z_1)\oplus \O_\C(z_2))$, and all the other sections have self-intersection number at least $2$. Therefore $\seg(\PP(\O_\C(z_1)\oplus \O_\C(z_2)))=0$. Finally, a section $\s$ of self-intersection $2$ passing through $q_1$ is the strict transform of a constant section in $\C\times \PP^1$ which also intersects the fiber of $p_2$. Thus $\s$ also passes through $q_2$.
	\end{enumerate}
\end{proof}

\begin{proposition}\label{segAtiyah}
	Let $\C$ be an elliptic curve. Let $p_1\in \C\times \PP^1$ such that $\pi(p_1)=z_1$ and let $\epsilon_{p_1}:\C\times \PP^1 \dashrightarrow \PP(\O_\C(z_1)\oplus \O_\C)$. We denote by $q_1$ be the unique base point of $\epsilon_{p_1}^{-1}$. Then the following hold:
	\begin{enumerate}
		\item For all $p_2\in \PP(\O_\C(z_1)\oplus \O_\C)$ in the same fiber as $q_1$, such that $p_2\neq q_1$ and $p_2$ does not belong to the unique $(-1)$-section of $\PP(\O_\C(z_1)\oplus \O_\C)$, the surface $\epsilon_{p_2}\epsilon_{p_1}(\C\times \PP^1)$ is $\C$-isomorphic to $\A_{0,\C}$. Moreover, $\A_{0,\C}$ has a unique section $\sigma_{0}$ of self-intersection number $0$ and all the other sections have self-intersection number at least $2$. In particular $\seg(\A_{0,\C})=0$.
		\item For all $p_3 \in \A_{0,\C}\setminus \sigma_0$, the surface $\epsilon_{p_3}(\A_{0,\C})$ is $\C$-isomorphic to $\A_{1,\C}$ and $\seg(\A_{1,\C})=1$.
	\end{enumerate} 
\end{proposition}

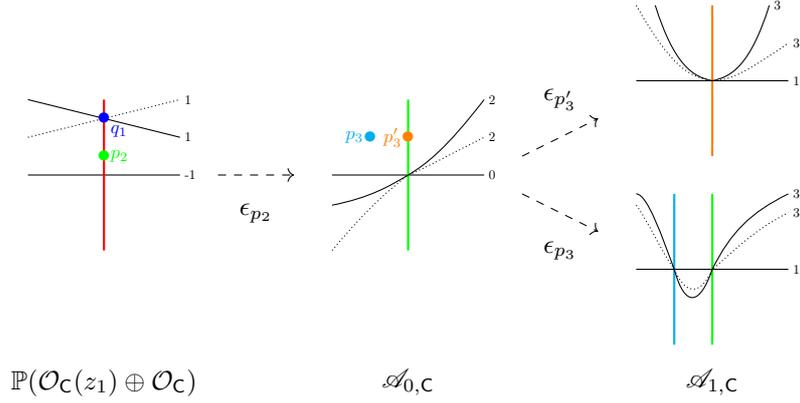
\begin{center}
	\begin{tikzpicture}[scale=0.5]
	\draw (0,1.5) -- (4,1.5);	\draw (4,1.5) node[scale=0.5,right]{-1};
	\draw[thick, red] (2,-0.5) -- (2,3.5);
	\draw[densely dotted] (0,2.5) -- (4,3.5);\draw (4,3.5) node[scale=0.5,right]{1};
	\draw (0,3.5) -- (4,2.5);\draw (4,2.5) node[scale=0.5,right]{1};
	\draw[blue] (2,3) node[below right,scale=0.7]{$q_1$} node{$\bullet$};
	\draw[green] (2,2) node[right,scale=0.7]{$p_2$} node{$\bullet$};
	\draw (2,-4) node[scale=1]{$\PP(\O_\C(z_1)\oplus \O_\C)$};
	
	\draw [dashed,->] (5,1.5) -- (7,1.5);
	\draw (6,0.5) node[scale=1]{$\epsilon_{p_2}$};
	
	\draw (8,1.5) -- (12,1.5);	\draw (12,1.5) node[scale=0.5,right]{0};
	\draw[thick, green] (10,-0.5) -- (10,3.5);
	\draw[cyan] (9,2.5) node[left,scale=0.7]{$p_3$} node{$\bullet$};
	\draw (8,0.7) .. controls (9,0.9) and (9.5,1.2) ..(10,1.5) .. controls (10.5,1.8) and (11,2.2) .. (12,3.5); \draw (12,3.5) node[scale=0.5,right]{2};
	\draw[densely dotted] (8,-0.5) .. controls (9,0.7) and (9.5,1.1) ..(10,1.5) .. controls (10.5,1.7) and (11,2) .. (12,2.5);\draw (12,2.5) node[scale=0.5,right]{2};
	\draw (10,-4) node[scale=1]{$\A_{0,\C}$};
	\draw[orange] (10,2.5) node[left,scale=0.7]{$p'_3$} node{$\bullet$};
	
	\draw [dashed,->] (13,1) -- (15,0);
	\draw (14,-0.5) node[scale=1]{$\epsilon_{p_3}$};
	\draw [dashed,->] (13,2) -- (15,3);
	\draw (14,3.5) node[scale=1]{$\epsilon_{p'_3}$};
	
	\draw[thick, cyan] (17,-3) -- (17,1);
	\draw[thick, green] (18,-3) -- (18,1);
	\draw (16,-1) -- (20,-1);	\draw (20,-1) node[scale=0.5,right]{1};
	\draw (16,1) .. controls (16.3,1) and (16.7,-0.2) .. (17,-1) .. controls (17.3,-2.1) and (17.7,-1.9) .. (18,-1) .. controls (18.5,0) and (19,0.5) .. (20,1);\draw (20,1) node[scale=0.5,right]{3};
	\draw[densely dotted] (16,0.7) .. controls (16.3,0.3) and (16.7,-0.6) ..(17,-1) .. controls (17.3,-1.7) and (17.7,-1.7) .. (18,-1) .. controls (18.5,-0.6) and (19,0) .. (20,0.5);\draw (20,0.5) node[scale=0.5,right]{3};
	\draw (18,-4) node[scale=1]{$\A_{1,\C}$};
	
	\draw[thick, orange] (18,2) -- (18,6);
	\draw (16,4) -- (20,4);	\draw (20,4) node[scale=0.5,right]{1};
	\draw (16.5,6) .. controls (17,4.4) and (17.5,4.1) .. (18,4) .. controls (18.5,4.1) and (19,4.4) .. (19.5,6);\draw (19.5,6) node[scale=0.5,right]{3};
	\draw[densely dotted]  (16,6) .. controls (17,4.2) and (17.5,4.1) .. (18,4) .. controls (18.5,4) and (19,4.3) .. (20,5);\draw (20,5) node[scale=0.5,right]{3};

	\end{tikzpicture}
\end{center}	
	
\begin{proof}
	We prove $(1)$ and $(2)$ simultaneously. Let $p_2\in \PP(\O_\C(z_1)\oplus \O_\C)$ be as in the statement of the proposition. Because $\PP(\O_\C(z_1)\oplus \O_\C)$ has a unique $(-1)$-section which does not contain $p_2$ and all sections of self-intersection number $1$ pass through $q_1$, the strict transform under $\epsilon_{p_2}$ of the unique $(-1)$-section is the unique section $\sigma_0$ of self-intersection number $0$ and all the other sections have self-intersection at least $2$. In particular, $\seg(\epsilon_{p_2}\epsilon_{p_1}(\C\times \PP^1))=0$ and $\epsilon_{p_2}\epsilon_{p_1}(\C\times \PP^1)$ is indecomposable from Proposition \ref{numberminimalsection} (3)(ii). Then for all $p_3\in \epsilon_{p_2}\epsilon_{p_1}(\C\times \PP^1) \setminus \sigma_0$, we have $\seg(\epsilon_{p_3}\epsilon_{p_2}\epsilon_{p_1}(\C\times \PP^1))=1$ so $\epsilon_{p_3}\epsilon_{p_2}\epsilon_{p_1}(\C\times \PP^1)$ is indecomposable from Proposition \ref{numberminimalsection} (1). At this stage, we know that $\epsilon_{p_2}\epsilon_{p_1}(\C\times \PP^1)$ and $\epsilon_{p_3}\epsilon_{p_2}\epsilon_{p_1}(\C\times \PP^1)$ are the Atiyah $\PP^1$-bundles defined in Theorem \ref{Atiyah}. Moreover, we know that $\A_{0,\C}=\PP(\mathcal{F}_{0,\C})$ and the indecomposable vector bundle $\mathcal{F}_{0,\C}$ of rank $2$ and degree $0$ satisfies:
	\[
	0 \to \O_\C\to \mathcal{F}_{0,\C} \to \O_\C \to 0.
	\]	
	From Corollary \ref{Segre}, we have that $\seg(\A_{0,\C}) = \deg(\mathcal{F}_{0,\C}) - 2\deg(\M)=-2\deg(\M)$ where $\M$ is a maximal line subbundle of $\mathcal{F}_{0,\C}$. Since $\O_\C$ is a line subbundle of $\mathcal{F}_{0,\C}$, it follows that $\deg(\M) \geq 0$ and $\seg(\A_{0,\C}) \leq 0$. Therefore $\epsilon_{p_2}\epsilon_{p_1}(\C\times \PP^1)\simeq \A_{0,\C}$ and $\epsilon_{p_3}\epsilon_{p_2}\epsilon_{p_1}(\C\times \PP^1) \simeq \A_{1,\C}$.
\end{proof}

\section{The classification}

\subsection{Infinite inclusion chains of automorphism groups}

In this subsection, we prove Theorem \ref{A}. Let $\S$ be a $\PP^1$-bundle over a curve $\C$ of genus $\g$. If $\g\geq 2$, it is known that $\Autzero(\C)$ is trivial (see e.g.\ \cite[Exercise IV.2.5]{Hartshorne}) and it implies that $\Autzero(\S)$ is a subgroup of $\Aut_\C(\S)$. When $\g=1$ and $\seg(\S)<0$, it is still true that $\Autzero(\S)$ is a subgroup of $\Aut_\C(\S)$ by the following result: 

\begin{lemma}\cite[Lemma 7]{Maruyama} \label{autoellipticseg<0}
	If $\S$ is a $\PP^1$-bundle over an elliptic curve $\C$ with $\seg(\S)<0$ then the image of $\Aut(\S) \to \Aut(\C)$ is a finite group.
\end{lemma}

In \cite[Theorem 2]{Maruyama}, the $\C$-automorphism groups of ruled surfaces over $\C$ are classified. We will not need the entire classification but we will use:

\begin{lemma}\cite[Theorem 2 (1) and case (b) p.92]{Maruyama} \label{automaruyama}
	Let $\S=\PP(\V)$ be a $\PP^1$-bundle over a curve $\C$, let $\s$ be a section of $\S$ and $\L(\s)$ be the line subbundle of $\V$ associated to $\s$ (see Lemma \ref{sub<->section}). We choose trivializations of $S$ such that $\sigma$ is the infinite section (Lemma \ref{automorphismssections} (1)). The following holds true:
	\begin{enumerate}
		\item If $\seg(\S)>0$, then $\Aut_\C(\S)$ is finite. 
		\item If $\seg(\S)< 0$ and $\gamma\in \Gamma(\C,\det(\V)^{-1} \otimes \L(\s)^2)$, then the local isomorphisms:
		\begin{align*}
		U_i \times \PP^1 & \to U_i\times \PP^1 \\
		(x,[u:v]) & \mapsto
		\left(x,
		\begin{bmatrix}
		1 & \gamma_{|U_i} \\
		0 & 1
		\end{bmatrix}
		\cdot 
		\begin{bmatrix}
			u \\ v
		\end{bmatrix}
		\right)
		\end{align*}
		glue into a $\C$-automorphism $f_\gamma$ of $\S$. 
	\end{enumerate}
\end{lemma}

\begin{remark}
	The automorphism $f_\gamma$ plays a crucial role in the proof of Theorem \ref{A}, hence we recall Maruyama's construction. First we write the transition maps of $S$ as $s_{ij}\colon U_j\times \PP^1 \dashrightarrow U_i\times \PP^1$, $(x,[u:v])\mapsto (x,[a_{ij}u+c_{ij}v:b_{ij}v])$ where $a_{ij}$ are the transition maps of the line bundle $\L(\sigma)$. Then $b_{ij}$ are the transition maps of the line bundle $\det(V)\otimes \L(\sigma)^{-1}$. The local isomorphisms $f_{\gamma_i}\colon U_i\times \PP^1 \to U_i\times \PP^1$, $(x,[u:v])\mapsto (x,[u+\gamma_iv:v])$, where $\gamma_i\in \O_\C(U_i)$, glue into a $\C$-automorphism of $S$ if and only if $s_{ij} f_{\gamma_j} = f_{\gamma_i} s_{ij}$, and a direct computation shows that it is equivalent to the condition $a_{ij}b_{ij}^{-1}\gamma_j=\gamma_i$. In particular, $(\gamma_i)_i$ defines a section of the line bundle $\det(\V)^{-1}\otimes \L(\sigma)^2$. 
\end{remark}

\begin{proof}[\bf{Proof of Theorem A}]
	Assume first that $\g\geq 2$ and $\seg(\S) > 0$, then it follows from Lemma \ref{automaruyama} (1) that $\Autzero(\S)$ is trivial. If $\g\geq 2$ and $\seg(\S)=0$ then from Lemma \ref{numberminimalsection} (3)(ii) and (iii), we know that $\S$ has at most two minimal sections because $\S$ is not trivial. Let $p$ be a point on a minimal section, then every automorphism of $\Autzero(\S)$ has to fix $p$ because $\Autzero(\S)$ is connected and $\Autzero(\C)$ is trivial. Therefore the elementary transformation $\epsilon_p:\S\dashrightarrow \T$ is $\Autzero(\S)$-equivariant, i.e.\ $\epsilon_p \Autzero(\S)\epsilon_p^{-1} \subset \Autzero(\T)$, and we have $\seg(\T)=-1$. So when $\g\geq 2$, it suffices to prove the theorem when $\seg(\S)<0$. In the statement of Theorem \ref{A}, we suppose that $\g\geq 2$, or $\g=1$ and $\seg(\S)<0$. From now on we assume that $\seg(\S)<0$ and $\g\geq 1$. Then from Lemma \ref{numberminimalsection} (2), the ruled surface $\S$ has a unique minimal section $\s$ and from Lemma \ref{autoellipticseg<0}, the algebraic group $\Autzero(\S)$ is a subgroup of $\AutC(\S)$. So any point $p_0$ of $\s$ is fixed by the action of $\Autzero(\S)$ and it implies that $\epsilon_{p_0}:\S \dashrightarrow \S_1$ is $\Autzero(\S)$-equivariant. By induction, there exist $\PP^1$-bundles $\S_n$ having a unique minimal section $\s_n$ and $p_n$ on $\s_n$ such that $\epsilon_{p_n}:\S_n \dashrightarrow \S_{n+1}$ is $\Autzero(\S_n)$-equivariant. By denoting $\phi_n=\epsilon_{p_{n-1}} ... \epsilon_{p_1}\epsilon_{p_0}$, we get a sequence:
	\[
	\Autzero(\S) \subseteq \phi_1^{-1}\Autzero(\S_1)\phi_1 \subseteq ... \subseteq \phi_n^{-1}\Autzero(\S_n)\phi_n \subseteq ...
	\]
	and it remains to prove that the obtained sequence is not stationary. Suppose $\S_{n+1} = \PP(\V_{n+1})$, we define $\L_{n+1}= \det(\V_{n+1})^{-1}\otimes \L(\s_{n+1})^{2}$. Let $q_n$ be the unique base point of $\epsilon_{p_n}^{-1}$ and we can assume that $q_n = (z,[0:1])$ over a trivializing open subset. If $n$ is large enough then $\deg(\L_{n+1})=-\seg(\S_{n+1})$ is large enough, and it implies that $\h^1(\C,\L_{n+1})=\h^1(\C,\L_{n+1} - z)=0$ by Serre duality. From Riemann-Roch theorem: $\h^0(\C, \L_{n+1}-z) = \deg(\L_{n+1})-g < \deg( \L_{n+1})-g+1 = \h^0(\C, \L_{n+1})$, and therefore $z$ is not a base point of the complete linear system $\vert \L_{n+1} \vert$.
	Then there exists $\gamma \in \Gamma(\C,\L_{n+1})$ such that $\gamma(z)\neq 0$, i.e.\ $f_\gamma$
	defines an automorphism of $\S_{n+1}$ which does not fix $q_n$. In consequence, the automorphism $f_\gamma$ defined in Lemma \ref{automaruyama} $(2)$ does not belong to $\epsilon_n \Autzero(\S_n) \epsilon_n^{-1}$ and $\epsilon_n\Autzero(\S_n)\epsilon_n^{-1}\neq \Autzero(\S_{n+1})$ when $n$ is taken large enough.
\end{proof}

\subsection{Maximal connected algebraic subgroups of $\Bir(\C\times \PP^1)$}

In this subsection, we prove Theorem \ref{B}. From Theorem \ref{A}, we know that a maximal connected algebraic subgroup of $\Bir(\C\times \PP^1)$ is conjugate to $\Autzero(\C\times \PP^1)\simeq \PGL_2$ if $\g\geq 2$ (and it is maximal from Lemma \ref{maximal} (3)), or $\Autzero(\S)$ for some ruled surface $\S$ with $\seg(\S)\geq 0$ if $\g=1$. The following lemma determines the remaining cases when $\g=1$:

\begin{lemma}\label{remainingcases}
	Let $\S$ be a ruled surface over an elliptic curve $\C$. If $\seg(\S)\geq 0$ then $\S$ is isomorphic to one of the following:
	\begin{enumerate}
		\item $\C\times \PP^1$,
		\item $\A_{0,\C}$,
		\item $\A_{1,\C}$,
		\item $\S_{z_1,z_2}$ for some distinct points $z_1,z_2\in \C$. 
	\end{enumerate} 
	Let $z_1,z_2\in \C$ be distinct points, then the surfaces $\C\times\PP^1$, $\A_{0,\C}$, $\A_{1,\C}$ and $\S_{z_1,z_2}$ are pairwise non-isomorphic.
\end{lemma}

\begin{proof}
	If $\S$ is indecomposable, it follows from Theorem \ref{Atiyah} that $\S$ is isomorphic to $\A_{0,\C}$ or $\A_{1,\C}$, and their Segre invariant is non-negative (Proposition \ref{segAtiyah}). If $\S$ is decomposable and $\seg(\S)\geq 0$, it follows from Proposition \ref{numberminimalsection} that $\seg(\S)=0$ and $\S$ is $\C$-isomorphic to $\PP(\L(\s_1)\oplus \L(\s_2))$ where $\s_1$ and $\s_2$ are disjoint minimal sections. Tensoring $\L(\s_1)\oplus \L(\s_2)$ by a line bundle with degree $(-\deg(\L(\s_1)) +1)$, it follows that $\S$ is $\C$-isomorphic to $\PP(\L_1\oplus \L_2)$ where $\L_1$ and $\L_2$ are line bundles of degree $1$. Since $\C$ is an elliptic curve, the line bundles $\L_1$ and $\L_2$ are respectively isomorphic to $\O_\C(z_1)$ and $\O_\C(z_2)$ for some $z_1,z_2\in \C$. Indeed, a line bundle of degree $1$ over $\C$ corresponds to a divisor of degree $1$ on $\C$, and its complete linear system is a unique point by Riemann-Roch formula. If $z_1=z_2$ then $\S$ is isomorphic to $\C\times \PP^1$, otherwise $\S$ is isomorphic to $\S_{z_1,z_2}$. Finally, the Atiyah ruled surfaces are not isomorphic to each other from Theorem \ref{Atiyah}, and they cannot be isomorphic to a decomposable $\PP^1$-bundle. Since the surface $\S_{z_1,z_2}$ has exactly two sections of self-intersection $0$ from Proposition \ref{decomposable} (2), it cannot be isomorphic to $\C\times \PP^1$. 
\end{proof}

If $\S$ is isomorphic to $\C\times \PP^1$ with $\C$ an elliptic curve, then $\Autzero(\S)\simeq \C\times \PGL_2$ is a maximal algebraic subgroup of $\Bir(\C\times \PP^1)$. From Proposition \ref{minimalsurfaces}, Lemma \ref{remainingcases} and Theorem \ref{A}, we are left with studying the maximality of $\Autzero(\S)$ when $\S$ is $\A_{0,\C}$, $\A_{1,\C}$ and $\S_{z_1,z_2}$ for $z_1,z_2\in \C$.

\begin{lemma}\label{autodescend}
	Let $\C$ be a curve of genus $\g\geq 1$, let $\pi\colon \S\to \C$ and $\pi'\colon \S'\to \C$ be ruled surfaces. Then an isomorphism from $\S$ to $\S'$ induces an automorphism of $\C$. If $\S=\S'$, we have a morphism of group schemes: 
	\[
	\pi_* \colon  \Aut(\S) \to \Aut(\C).
	\]
	The restriction of $\pi_*$ to the connected components of identity coincides with the morphism of algebraic groups $\Autzero(\S)\to \Autzero(\C)$ induced by Blanchard's Lemma.
\end{lemma}

\begin{proof}
	If $g\colon \S\to \S'$ is an isomorphism and $f$ is a fiber in $\S$, then $\pi' g_{|f}$ is a morphism from $f\simeq \PP^1$ to $\C$. Hence it is constant and the image of $f$ by $g$ is a fiber $f'$ in $\S'$. Then the isomorphism $g$ induces an automorphism of $\C$. If $\S=\S'$, we get a morphism $\pi_*\colon \Aut(\S)\to \sf{Bij}(\C)$, where $\sf{Bij}(\C)$ denotes the set of bijections of $\C$. Let $\sigma$ be a section of $\pi$ and $g\in \Aut(S)$. Then $\pi_*(g) = \pi g \sigma$, and in particular $\pi_*(g)$ is a morphism. Since $g$ is an automorphism, it follows that $\pi_*(g)$ is also an automorphism and the image of $\pi_*$ is contained in $\Aut(\C)$. The restriction of $\pi_*$ induces a morphism of algebraic group $\Autzero(\S)\to \Autzero(\C)$. Then $\pi$ is $\Autzero(\S)$-equivariant, with $\Autzero(\S)$ acting on $\C$ by $(g,x)\mapsto \pi_*(g)(x)$. In particular, $\pi_*$ coincides with the morphism induced by Blanchard's Lemma by the unicity part of Lemma \ref{Blanchard}.
\end{proof}

The following proposition will be useful.

\begin{proposition}\label{Briontrick}
	Let $\C$ be a curve and $\pi\colon \S\to \C$ be a $\PP^1$-bundle. Then $\Aut(\S)$ is an algebraic group.
\end{proposition}

\begin{proof}
	Since $\S$ is a ruled surface, the adjunction formula gives $-\mathrm{K}_\S \cdot f = 2$ for all fibers $f$. In particular, $-\mathrm{K}_\S$ is $\pi$-ample and if $A$ denotes an ample divisor on $\C$, then the divisor $D= -\mathrm{K}_\S + m\pi^*(A)$ is ample for $m$ positive and large enough (see e.g.\ \cite[\href{https://stacks.math.columbia.edu/tag/0892}{Lemma 0892 (1)}]{stacks-project}, but it is also a consequence of Nakai ampleness criterion). Moreover, the numerical class of $D$ is fixed by $\Aut(\S)$ since $\mathrm{K}_\S$ and $\pi^*(A)$ are fixed. From \cite[Theorem 2.10]{Brionnotes}, the group scheme $\Aut(\S)$ has finitely many connected components and thus it is an algebraic group.
\end{proof} 

We will use the following proposition to show that the automorphism groups of the Atiyah's ruled surfaces are maximal.  
  
\begin{proposition}\label{Atiyahmax}
	Let $\C$ be an elliptic curve and $\pi\colon \A_{i,\C} \to \C$ be the structure morphism. For $i\in \{0,1\}$, the morphism of algebraic groups induced by Blanchard's Lemma (or by Lemma \ref{autodescend}):
	\[
	\pi_*\colon \Autzero(\A_{i,\C}) \rightarrow \Autzero(\C)
	\]
	is surjective.
\end{proposition}	

\begin{proof}
	Let $g\in \Aut(\C)$. Then the pullback $\pi^*\colon g^*(\A_{i,\C})\to \C$ is an indecomposable $\PP^1$-bundle over $\C$. Since the Atiyah bundles are unique up to $\C$-isomorphism from Theorem \ref{Atiyah}, it follows that $g^*(\A_{i,\C})$ is $\C$-isomorphic to $\A_{i,\C}$, so the following diagram is commutative:
	\[
	\begin{tikzcd} [row sep = 3em]
	\A_{i,\C} \arrow[r,"\sim"] \arrow[rr, bend left,"\tilde{g}"]\arrow[rd,"\pi" swap]&  g^*(\A_{i,\C})  \arrow[r]\arrow[d,"\pi^*" left] & \A_{i,\C} \arrow[d,"\pi"] \\
	&  \C  \arrow[r,"g" below]&  \C.
	\end{tikzcd}
	\] 
	In particular, there exists $\tilde{g}\in \Aut(\A_{i,\C})$ such that $\pi_*(\tilde{g})=g$, i.e.\ the morphism $\pi_*\colon \Aut(\A_{i,\C})\to \Aut(\C)$ from Lemma \ref{autodescend} is surjective. Let $H={(\pi_*)}^{-1}(\Autzero(\C))$, then we can write $H = \bigsqcup_{j\in J} H_j$ where $H_j$ are the connected components of $H$ and $J$ is finite from Proposition \ref{Briontrick}. Because $H$ contains $\Autzero(\A_{i,\C})$, we can assume that $H_0= \Autzero(\A_{i,\C})$. Then $\pi_*(H_0)$ is a connected algebraic subgroup of $\Autzero(\C)\simeq \C$ and hence it has to be $\Autzero(\C)$ or a point. If $\pi_*(H_0)$ is a point then $\pi_*(H_j)$ is also a point because $h_j\cdot H_0 = H_j$ for all $h_j\in H_j$. Then $\pi_*(H)$ is finite and it is a contradiction because $\Autzero(\C)$ is infinite. In consequence $\pi_*(H_0)= \Autzero(\C)$ i.e.\ $\pi_*$ induces a surjective morphism of algebraic groups $\Autzero(\A_{i,\C}) \to \Autzero(\C)$.
\end{proof}	  
  
We will use Proposition \ref{decomposablemax} to show that for distinct points $z_1,z_2\in \C$, the algebraic subgroup $\Autzero(\S_{z_1,z_2})$ is maximal. To prove Proposition \ref{decomposablemax}, we first prove the following lemma: 
  
\begin{lemma}\label{explicitelem}
	Let $\C$ be an elliptic curve and $f\in \kk(\C)^{*}$ such that $\div(f)=y_1 + z_1 -y_2 - z_2$ with $y_1,y_2,z_1,z_2$ distinct points of $\C$. We define:
	\begin{align*}
		\phi_f \colon \C\times \PP^1 & \dashrightarrow \C\times \PP^1 \\
		(x,[u:v]) &\longmapsto (x,[f(x)u:v]).
	\end{align*}
	Then $\phi_f$ is the birational map consisting in the blowup of $\C\times \PP^1$ at $p_1=(y_1,[1:0]),q_1=(z_1,[1:0]),p_2=(y_2,[0:1]),q_2=(z_2,[0:1])$; followed by the contraction of the strict transform of their fibers.
\end{lemma}

\begin{proof}
	First $\phi_f$ is birational because $\phi_f^{-1} = \phi_{1/f}$. The base points of $\phi_f$ are exactly $p_1$, $q_1$, $p_2$ and $q_2$ and have all order $1$, so one can check by blowups in local charts that $\phi_f$ corresponds to the blowups at $p_1$, $p_2$, $q_1$, $q_2$ followed by the contraction of the strict transform of their fibers.
\end{proof}

\begin{proposition}\label{decomposablemax}
	Let $\C$ be an elliptic curve, let $z_1,z_2\in \C$ be distinct points and $t$ be a translation of $\C$. Then $\S_{z_1,z_2}$ is $\C$-isomorphic to $\S_{t(z_1),t(z_2)}$ and moreover, the morphism of algebraic groups induced by Blanchard's Lemma (or by Lemma \ref{autodescend}):
	\[
	\pi_*\colon \Autzero(\S_{z_1,z_2}) \to \Autzero(\C)
	\]
	is surjective.
\end{proposition}

\begin{proof}
	As $z_1-z_2$ is linearly equivalent to $t(z_1)-t(z_2)$, there exists $f\in \kk(\C)^{*}$ such that $\div(f) = z_1+t(z_2) - t(z_1)-z_2 $. We define $\phi_f\colon \C\times \PP^1 \dashrightarrow \C\times \PP^1$ as in Lemma \ref{explicitelem}, and we know that $\phi_f=\kappa \beta^{-1}$, where $\beta$ is the blowup $\beta\colon \X\to \C\times \PP^1$ at $p_1=(z_1,[1:0]),q_1=(t(z_1),[1:0]),p_2=(z_2,[0:1]),q_2=(t(z_2),[0:1])$ and $\kappa\colon \X\to \C\times \PP^1$ is the contraction of the strict transform of their fibers. Let $\E_{q_1}$ and $\E_{q_2}$ be the exceptional divisors from respectively the blowups of $q_1$ and $q_2$, and let $\tilde{f}_{p_1}$ and $\tilde{f}_{p_2}$ be strict transforms under $\beta^{-1}$ of the fibers ${f}_{p_1}$ and ${f}_{p_2}$ containing respectively $p_1$ and $p_2$. We denote by $\xi\colon \X\to \S$ the contraction of $\E_{q_1}$, $\E_{q_2}$, $\tilde{f}_{p_1}$ and $\tilde{f}_{p_2}$, i.e.\ $\epsilon_{p_1,p_2} = \xi \beta^{-1}$. Denote by $\tilde{p}_1,\tilde{p}_2,\tilde{q}_1,\tilde{q}_2$ the base points of $\phi_f^{-1}$, respectively from the elementary transformations centered at $p_1,p_2,q_1,q_2$. Similarly, we have $\epsilon_{\tilde{q}_1,\tilde{q}_2}=\xi \kappa^{-1}$, and the following diagram is commutative:
	
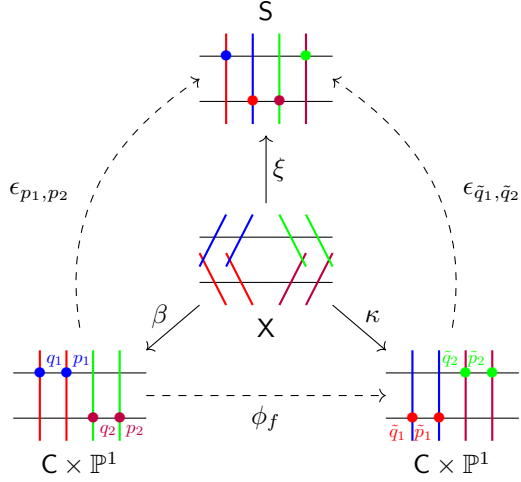
\begin{center}
		\begin{tikzpicture}[xscale=0.35,yscale=0.3]
			\draw (0,0) -- (5,0);
			\draw (0,2) -- (5,2);
			\draw[red,thick] (1,-1) -- (1,3);\draw[blue] (1,2) node[above right, scale=0.7]{$q_1$} node{$\bullet$}; 
			\draw[red,thick] (2,-1) -- (2,3);\draw[blue] (2,2) node[above right, scale=0.7]{$p_1$} node{$\bullet$}; 
			\draw[green,thick] (3,-1) -- (3,3);\draw[purple] (3,0) node[below right, scale=0.7]{$q_2$} node{$\bullet$};
			\draw[green,thick] (4,-1) -- (4,3);\draw[purple] (4,0) node[below right, scale=0.7]{$p_2$} node{$\bullet$};
			\draw (2.5,-2) node[scale=1]{$\C\times \PP^1$};
			
			\draw [dashed,->] (2.5,4) to [bend left] (7,15);\draw (1,10) node[scale=1]{$\epsilon_{p_1,p_2}$};		
			\draw [<-] (5,3) -- (7,5);\draw (5.5,4.5) node[scale=1]{$\beta$};
			\draw [dashed,->] (5,1) -- (14,1);\draw (9.5,0) node[scale=1]{$\phi_f$};
			
			\draw (7,6) -- (12,6);
			\draw (7,8) -- (12,8);
			\draw[red,thick] (8,5) -- (7,7.3);
			\draw[red,thick] (9,5) -- (8,7.3);
			\draw[purple,thick] (10,5) -- (11,7.3);
			\draw[purple,thick] (11,5) -- (12,7.3);
			\draw[blue,thick] (9,9) -- (8,6.7);
			\draw[blue,thick] (8,9) -- (7,6.7);
			\draw[green,thick] (10,9) -- (11,6.7);
			\draw[green,thick] (11,9) -- (12,6.7);
			\draw (9.5,4) node[scale=1]{$\X$};
			
			\draw [->] (9.5,9.5) -- (9.5,12.5); \draw (10,11) node[scale=1]{$\xi$};
			\draw [->] (12,5) -- (14,3);\draw (13.5,4.5) node[scale=1]{$\kappa$};
			\draw [dashed,->] (16.5,4) to [bend right] (12,15);\draw (18,10) node[scale=1]{$\epsilon_{\tilde{q}_1,\tilde{q}_2}$};	
			
			\draw (14,0) -- (19,0);
			\draw (14,2) -- (19,2);
			\draw[blue,thick] (15,-1) -- (15,3);\draw[red] (15,0) node[below left, scale=0.7]{$\tilde{q}_1$}  node{$\bullet$};
			\draw[blue,thick] (16,-1) -- (16,3);\draw[red] (16,0) node[below left, scale=0.7]{$\tilde{p}_1$}  node{$\bullet$};
			\draw[purple,thick] (17,-1) -- (17,3);\draw[green] (17,2) node[above left, scale=0.7]{$\tilde{q}_2$}  node{$\bullet$};
			\draw[purple,thick] (18,-1) -- (18,3);\draw[green] (18,2)  node[above left, scale=0.7]{$\tilde{p}_2$}  node{$\bullet$};
			\draw (16.5,-2) node[scale=1]{$\C\times \PP^1$};
			
			\draw (7,14) -- (12,14);
			\draw (7,16) -- (12,16);
			\draw[red,thick] (8,13) -- (8,17);\draw[blue] (8,16) node{$\bullet$};
			\draw[blue,thick] (9,13) -- (9,17);\draw[red] (9,14) node{$\bullet$};			
			\draw[green,thick] (10,13) -- (10,17);\draw[purple] (10,14) node{$\bullet$};			
	 		\draw[purple,thick] (11,13) -- (11,17);\draw[green] (11,16) node{$\bullet$};
	 		\draw (9.5,18) node[scale=1]{$\S$};
		\end{tikzpicture}
	\end{center}
 Therefore the surfaces $\S$, $\S_{z_1,z_2}$ and $\S_{t(z_1),t(z_2)}$ are  $\C$-isomorphic. It implies in particular that every translation of $\C$ can be lifted to an automorphism of $\S_{z_1,z_2}$. Therefore, we have a morphism of algebraic groups $\pi_*\colon \Aut(\S_{z_1,z_2}) \to \Aut(\C)$ such that $\Autzero(\C)$ is contained in the image of $\pi_*$. The proof ends in the same way as the proof of Proposition \ref{Atiyahmax}. Let $H={(\pi_*)}^{-1}(\Autzero(\C))$, then we can write $H=\bigsqcup_{j\in J} H_j$ where $H_j$ are the connected components of $H$ and $J$ is finite from Proposition \ref{Briontrick}. The image of $\Autzero(\S_{z_1,z_2})$ by $\pi_*$ cannot be a point because it would implies that the image of $H$ is finite and it is a contradiction, therefore $\pi_*(\Autzero(\S_{z_1,z_2}))=\Autzero(\C)$.
\end{proof}

Unlike the ruled surfaces $\A_{0,\C}$ and $\A_{1,\C}$ which are unique up to $\C$-isomorphism, the surfaces $\S_{z_1,z_2}$ depend on the choice of the points $z_1,z_2\in \C$. In Lemma \ref{isomorphicdecomposable}, we determine the $\C$-isomorphism classes in the family $\{\S_{z_1,z_2}\}_{z_1,z_2\in \C}$. In Lemma \ref{decomposableconjugate}, we determine the isomorphism classes in the family $\{\S_{z_1,z_2}\}_{z_1,z_2\in \C}$ and the conjugacy classes in the family $\{\Autzero(\S_{z_1,z_2})\}_{z_1,z_2\in \C}$ as algebraic subgroups of $\Bir(\C\times \PP^1)$.

\begin{lemma}\label{isomorphicdecomposable}
	Let $\C$ be an elliptic curve and $z_1,z_1',z_2,z_2'\in \C$ such that $z_1\neq z_2$ and $z_1'\neq z_2'$. Let $\pi\colon \S_{z_1,z_2} \to \C$ and $\pi'\colon \S_{z_1',z_2'} \to \C$ be the structure morphisms. Let $\s_1$ and $\s_2$ be the two disjoint sections of self-intersection $0$ in $\S_{z_1,z_2}$; and let $\s_1'$ and $\s_2'$ be the two disjoint sections of self-intersection $0$ in $\S_{z_1',z_2'}$. The following hold:
	\begin{enumerate}
		\item Let $q_1\in \s_1$. Then every section of $\S_{z_1,z_2}$ of self-intersection $2$ passing through $q_1$ also passes through the unique point $q_2\in \s_2$ such that $\pi(q_2) -\pi(q_1) = z_2 - z_1$. 
		\item Let $q_1,q_2 \in \S_{z_1,z_2}$ as in $(1)$. If there exist $q_1'\in \s_1'$, $q_2'\in \s_2'$ such that $\pi'(q_2')-\pi'(q_1')=z_2-z_1$ and if there exists a section of $\S_{z_1',z_2'}$ of self-intersection $2$ passing through $q_1'$ and $q_2'$, then $\S_{z_1',z_2'}$ is $\C$-isomorphic to $\S_{z_1,z_2}$.
		\item The ruled surfaces $\S_{z_1,z_2}$ and $\S_{z_1',z_2'}$ are $\C$-isomorphic if and only if there exists $t\in \Autzero(\C)$ such that $t(\{z_1,z_2\})=\{z_1',z_2'\}$.
	\end{enumerate}
\end{lemma}

\begin{proof}\
	\begin{enumerate}[wide]
		\item Let $p_1\in \s_1,p_2\in \s_2$ such that $\pi(p_1)=z_1$ and $\pi(p_2)=z_2$. Let $q_1\in \s_1,q_2\in \s_2$ and assume there exists a section $\s$ of $\S_{z_1,z_2}$ of self-intersection $2$ passing through $q_1$ and $q_2$. The translation of $\C$ sending $\pi(q_1)$ to $z_1$ lifts to $f\in \Autzero(\S_{z_1,z_2})$ (Proposition \ref{decomposablemax}), and since $\s_1$ is $\Autzero(\S_{z_1,z_2})$-invariant, $f$ sends $q_1$ to $p_1$. Then the section $f(\s)$ of self-intersection $2$ passes through $p_1$, and hence it also passes through $p_2$ from Proposition \ref{decomposable} $(2)$. Therefore, the automorphism $f$ sends $q_2$ to $p_2$ and $\pi(q_2)-\pi(q_1) = z_2-z_1$. In particular, all sections of self-intersection $2$ passing through $q_1$ also pass through $q_2$. 
		\item It follows from $(1)$ that $z_2'-z_1'=z_2-z_1$. Let $t$ be the translation of $\C$ by $z_1'-z_1$. Then $t(z_1)=z_1'$ and $t(z_2)=z_2'$. From Proposition \ref{decomposablemax}, the surfaces $\S_{z_1,z_2}$ and $\S_{z_1',z_2'}$ are $\C$-isomorphic.
		\item Assume that $\S_{z_1,z_2}$ is $\C$-isomorphic to $\S_{z_1',z_2'}$. From Proposition \ref{decomposablemax}, there exists $f\in \Autzero(\S_{z_1',z_2'})$ such that $\pi_*(f)$ is the translation of $\C$ sending $z_1'$ to $z_1$, and $\S_{z_1',z_2'}$ is $\C$-isomorphic to $\S_{z_1,z_2''}$ where $z_2''=\pi_*(f)(z_2')$. From $(1)$, we have that $z_2'' =z_2$ and this proves the direct implication. 
		Let $t\in \Autzero(\C)$ such that $t(\{z_1,z_2\})=\{z_1,z_2\}$. Without lost of generality, we can assume that $z_1'=t(z_1)$ and $z_2'=t(z_2)$. Then $\S_{z_1,z_2}$ is $\C$-isomorphic to $\S_{z_1',z_2'}$ from Proposition \ref{decomposablemax}.
	\end{enumerate}
\end{proof}

\begin{lemma}\label{decomposableconjugate}
	Let $\C$ be an elliptic curve. Denote by $\pi_1\colon \C\times \PP^1 \to \C$ the projection on the first factor and let $z_1,z_1',z_2,z_2'\in \C$ such that $z_1\neq z_2$, $z_1'\neq z_2'$. Let $p_1,p_1',p_2,p_2'\in \C\times \PP^1$ with $p_1,p_1'$ on the zero section, $p_2,p_2'$ on the infinite section and such that $\pi_1(p_1)=z_1$, $\pi_1(p_1')=z_1'$, $\pi_1(p_2)=z_2$, $\pi_1(p_2')=z_2'$. Then the following assertions are equivalent:
	\begin{enumerate}
		\item There exists $f\in \Aut(\C)$ such that $f(\{z_1,z_2\}) = \{z_1',z_2'\}$.
		\item The surfaces $\S_{z_1,z_2}$ and $\S_{z_1',z_2'}$ are isomorphic. 
		\item The algebraic subgroups $\Autzero(\S_{z_1,z_2})$ and $\Autzero(\S_{z_1',z_2'})$ are conjugate.
	\end{enumerate}
\end{lemma}

\begin{proof}\
		\begin{enumerate}[wide]
			\item[$(1)\Rightarrow (2)$] Assume there exists $f\in \Aut(\C)$ such that $f(\{z_1,z_2\}) = \{z_1',z_2'\}$ and $\pi \colon \S_{z_1,z_2} \to \C$ is the structure morphism. Then the ruled surface $f\pi : \S_{z_1,z_2} \to \C$ has $0$ as Segre invariant and is decomposable, thus it has two disjoint minimal sections $\s_1$ and $\s_2$. Let $q_1\in \s_1,q_2\in \s_2$ such that $f\pi(q_1)=z_1'$ and $f\pi(q_2)=z_2'$, then there is a section of self-intersection $2$ passing through $q_1$ and $q_2$.  From Lemma \ref{isomorphicdecomposable} $(2)$, the ruled surfaces $f\pi\colon \S_{z_1,z_2} \to \C$ and $\pi' \colon \S_{z_1',z_2'} \to \C$ are $\C$-isomorphic, and the following diagram is commutative:
			\[
			\begin{tikzcd} [row sep = 3em]
			 \S_{z_1,z_2} \arrow[rd,"f\pi",swap] \arrow[r,"\sim"]\arrow[d,"\pi" left] & \S_{z_1',z_2'} \arrow[d,"\pi' "] \\
			  \C  \arrow[r,"f" below]&  \C.
			\end{tikzcd}
			\] 
			Hence the surfaces $\S_{z_1,z_2}$ and $\S_{z_1',z_2'}$ are isomorphic. 
			\item[$(2)\Rightarrow (1)$] From Proposition \ref{decomposablemax}, there exists $f\in \Autzero(\S_{z_1,z_2})$ such that $\pi_*(f)$ is the translation from $z_1$ to $z_1'$. We can then assume that $z_1=z_1'$. From Lemma \ref{autodescend}, an isomorphism from $\S_{z_1,z_2}$ to $\S_{z_1,z_2'}$ induces an automorphism of $\C$. From Proposition \ref{decomposable} (2), this automorphism of $\C$ sends $z_2$ to $z_2'$ and fixes $z_1$, i.e.\ a group map with $z_1$ taken as the neutral element of the elliptic curve. Therefore, an isomorphism from $\S_{z_1,z_2}$ to $\S_{z_1',z_2'}$ induces an automorphism of $\C$ sending $z_1$ to $z_1'$ and $z_2$ to $z_2'$. 
			\item[$(2) \Rightarrow (3)$] If $\phi:\S_{z_1,z_2} \to \S_{z_1',z_2'}$ is an isomorphism then $\phi \Autzero(\S_{z_1,z_2}) \phi^{-1} = \Autzero(\S_{z_1',z_2'})$. 
			\item[$(3)\Rightarrow (2)$] Let $\phi:\S_{z_1,z_2} \to \S_{z_1',z_2'}$ be a birational map such that $\phi \Autzero(\S_{z_1,z_2}) \phi^{-1} = \Autzero(\S_{z_1',z_2'})$. Because $\Autzero(\S_{z_1,z_2})$ induces a transitive action on the base curve $\C$ from Proposition \ref{decomposablemax}, the action of $\Autzero(\S_{z_1,z_2})$ on $\S_{z_1,z_2}$ has no fixed points and it follows from Lemma \ref{maximal} $(3)$ that $\phi$ is an isomorphism.
		\end{enumerate}
\end{proof}

\begin{proof}[\bf{Proof of Theorem \ref{B}}]
	Let $\S$ be a ruled surface over $\C$. If $\g(\C)\geq 2$ or $\g(\C)=1$ and $\seg(\S)<0$, Theorem \ref{A} implies that every maximal connected algebraic subgroup of $\Bir(\S)$ is conjugate to $\Autzero(\C\times \PP^1)$ which is isomorphic to $\PGL_2$ if $\g\geq 2$, and isomorphic to $\C\times \PGL_2$ is $\g=1$. Therefore, the algebraic subgroup $\Autzero(\C\times \PP^1)$ is maximal from Lemma \ref{maximal} (3). We have proved in Lemma \ref{remainingcases} that it remains to consider the case $\g=1$ and show the maximality of $\Autzero(\S)$ when $\S$ is isomorphic to $\A_{0,\C}$, or $\A_{1,\C}$, or $\S_{z_1,z_2}$ for distinct points $z_1,z_2\in \C$. Assume $\S$ is one of these surfaces, then from Propositions \ref{Atiyahmax} and \ref{decomposablemax}, there is a surjective morphism of algebraic groups $\Autzero(\S)\to \Autzero(\C)$. Lemma \ref{maximal} $(3)$ implies that $\Autzero(\S)$ is maximal. Moreover, for distinct points $z_1,z_2\in \C$, the surfaces $\C\times \PP^1$, $\A_{0,\C}$, $\A_{1,\C}$ and $\S_{z_1,z_2}$ are not isomorphic to each other from Lemma \ref{remainingcases}. Hence the algebraic groups $\Autzero(\C\times \PP^1)$, $\Autzero(\A_{0,\C})$, $\Autzero(\A_{1,\C})$ and $\Autzero(\S_{z_1,z_2})$ are not conjugate to each other from Lemma \ref{maximal} $(3)$. Finally, Lemma \ref{decomposableconjugate} tells us that $\Autzero(\S_{z_1,z_2})$ is conjugate to $\Autzero(\S_{z_1',z_2'})$ if and only if there exists $f\in \Aut(\C)$ such that $f(\{z_1,z_2\})=\{z_1',z_2'\}$. 
\end{proof}

\subsection{Description of the maximal connected algebraic subgroups of $\Bir(\C\times \PP^1)$ as extensions}\label{subsectionextensions}

In this subsection, the curve $\C$ always denotes an elliptic curve and we prove Theorem \ref{Maruyamaext}.

\subsubsection{The algebraic groups $\Autzero(\S_{z_1,z_2})$}

\begin{proposition}\label{extensiondecomposable}
	Let $\C$ be an elliptic curve and $z_1,z_2$ be distinct points in $\C$. The group homomorphism $\pi_*$ of Lemma \ref{decomposablemax} gives rise to an exact sequence of algebraic groups:
	\[
	0 \to   \GG_m \to \Autzero(\S_{z_1,z_2}) \overset{\pi_*}{\to} \Autzero(\C) \to 0.
	\]
\end{proposition}

\begin{proof}
	From Proposition \ref{decomposablemax}, it suffices to prove that $\ker(\pi_*)\simeq \GG_m$. Let $p_1, p_2\in \C\times \PP^1$ respectively on the fibers of $z_1$ and $z_2$ such that $p_1$ is on the zero section and $p_2$ is on the infinite section; and let $\epsilon_{p_1,p_2}\colon \C\times \PP^1 \dashrightarrow \S_{z_1,z_2}$ be the blowups of $p_1,p_2$ followed by the contractions of their fibers. We denote respectively by $q_1$ and $q_2$ the base points of $\epsilon_{p_1,p_2}^{-1}$ which belong to the fibers of $z_1$ and $z_2$. The automorphisms $\phi_\alpha$ with $\alpha\neq 0$ defined by:
	\begin{align*}
		\phi_{\alpha}\colon \C\times \PP^1& \rightarrow \C\times \PP^1 \\
		\left( x,[u:v] \right)& \mapsto \left(x,[\alpha u :v]\right),
	\end{align*}
	form a subgroup of $\Aut(\C\times \PP^1)$ which we denote by $\Aut_{0,\infty}$. Since $p_1$ and $p_2$ are fixed by $\Aut_{0,\infty}$, the birational map $\epsilon_{p_1,p_2}$ is $\Aut_{0,\infty}$-equivariant i.e.\ $\GG_m \subset \ker(\pi_*)$. Conversely, let $f\in \Autzero(\S_{z_1,z_2})$ be such that $\pi_*(f)=id$. Then $f$ fixes $q_1$ and $q_2$ because they belong to one of the two minimal sections. Hence $\epsilon_{p_1,p_2}^{-1} f \epsilon_{p_1,p_2}$ is a $\C$-automorphism of $\C\times \PP^1$ which fixes $p_1$ and $p_2$, so it equals $\phi_\alpha$ for some $\alpha\in \GG_m$. Therefore $\ker(\pi_*)\simeq \GG_m$.
\end{proof}

\subsubsection{The algebraic group $\Autzero(\A_{0,\C})$}

\begin{lemma}\label{sectiontrivialbundleoverellipticcurve}
	Let $\C$ be an elliptic curve, $\pi_1\colon\C\times \PP^1\to \C$ be the projection on the first factor and $\s_\infty$ be the infinite section in $\C\times \PP^1$. For all $p,p'\in \s_\infty$, there exists a section $\s$ such that $\s^2=4$, $\s\cap \s_\infty = \{p,p'\}$ and they intersect transversely at $p$ and $p'$. Suppose $\s$ is such a section, then a section $\s'$ satisfies the same properties if and only if there exist $\alpha\in \GG_m$ and $\gamma\in \GG_a$ such that $\s'$ is the image of $\s$ by the automorphism:
	\begin{align*}
		\phi_{\alpha,\gamma}\colon \C\times \PP^1& \rightarrow \C\times \PP^1 \\
		\left( x,[u:v] \right)& \mapsto \left(x,[\alpha u + \gamma v:v]\right).
	\end{align*}
	Moreover $\phi_{\alpha,\gamma}$ is the unique $\C$-automorphism of $\C\times \PP^1$ which sends $\s$ to $\s'$.
\end{lemma}

\begin{proof}
	Let $z=\pi_1(p),z'=\pi_1(p')$ and $\D=z+z'$. From Riemann-Roch theorem and Serre duality, we have $\dim(\Gamma(\C,\D) ) = 2$. Because $1\in \Gamma(\C,\D)$, there is a section $\s \in \Gamma(\C,\D)$ with exactly two poles of order $1$ at $z$ and $z'$, i.e.\ $\s$ intersect transversely $\s_\infty$ at $p$ and $p'$, and $\{1,\s\}$ is a basis for $\Gamma(\C,\D)$. Since $\s$ is given by a morphism $g_\sigma:\C\to \PP^1$, we know from Lemma \ref{sectiontrivialbundle} that $\s^2=2\deg(g_\s)=4$. Let $\phi_{\alpha,\gamma}$ be an automorphism of $\C\times \PP^1$ defined as in the statement, then the section $\phi_{\alpha,\gamma}(\s)$ intersects transversely $\s_\infty$ at exactly $p$ and $p'$, and $\phi_{\alpha,\gamma}(\s)^2=4$. Conversely if $\s'$ is a section which satisfies the same properties, then $\s'\in \Gamma(\C,\D)$.  In particular if $\s\colon x\mapsto (x,[u(x):v(x)])$, then there exist $\alpha,\gamma \in \kk$ such that $\s'(x)=(x,[\alpha u(x)+\gamma v(x):v(x)]) = \phi_{\alpha,\gamma} \left(\s(x)\right)$. Finally, the $\C$-automorphisms of $\C\times \PP^1$ fixing the infinite section are of the form $\phi_{\alpha,\gamma}$ for some $\alpha\in \GG_m$ and $\gamma \in \GG_a$, and the image of $\s$ uniquely determines $\alpha$ and $\gamma$. Therefore $\phi_{\alpha,\gamma}$ is the unique $\C$-automorphism of $\C\times \PP^1$ sending $\s$ to $\s'$. 
\end{proof}

The group of all automorphisms $\phi_{\alpha,\gamma}$ is denoted $\Autinfini$ and it is isomorphic to $\GG_a\rtimes \GG_m$. In particular, $\Autinfini$ is connected.

\begin{lemma}\label{sectionsSp}
	Let $\C$ be an elliptic curve, $\s_\infty$ be the infinite section in $\C\times \PP^1$ and $\beta_p \colon\bup_p(\C\times \PP^1)\rightarrow \C\times \PP^1$ be the blowup of $p=(z,[1:0])$. Let $\tilde{\s}_\infty$ and $\tilde{f}_p$ be the strict transforms under $\beta_p^{-1}$ of respectively $\s_\infty$ and the fiber $f_p$ containing $p$ in $\C\times \PP^1$. Then $\beta_p^{-1}$ is $\Autinfini$-equivariant and $\beta_p^{-1} \Autinfini \beta_p$ induces a simply transitive action of $\GG_m$ on $\E_p \setminus \{\tilde{\s}_\infty,\tilde{f}_p\}$. More precisely, the following hold for all $b\in \tilde{\s}_\infty \setminus \E_p$:
	\begin{enumerate}
		\item There exists $e\in \E_p\setminus \{\tilde{f}_p,\tilde{\s}_\infty\}$ and a section $\s\subset \bup_p(\C\times \PP^1)$ of self-intersection $3$ passing through $b$ and $e$.
		\item For all $e'\in \E_p\setminus \{\tilde{f}_p,\tilde{\s}_\infty\}$ there exists a unique $\alpha\in \GG_m$ such that the sections of self-intersection $3$ and passing through $b$ and $e'$ are the image of $\s$ by the automorphism $\beta_p^{-1} \phi_{\alpha,\gamma} \beta_p$ for some $\gamma\in \GG_a$.
	\end{enumerate}
\end{lemma}

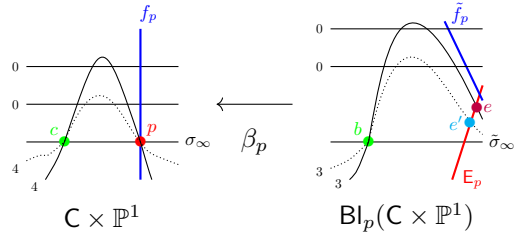
\begin{center}
	\begin{tikzpicture}[scale=0.5]
	\draw (0,1) -- (4,1); \draw (-0.5,1) node[scale=0.5,right]{0};
	\draw (0,0) -- (4,0); \draw (-0.5,0) node[scale=0.5,right]{0};
	\draw (0,-1) -- (4,-1); \draw (4,-1) node[scale=0.7,right]{$\s_\infty$};
	\draw[thick, blue] (3,-2) -- (3,2) ; \draw[blue] (2.8,2.4) node[scale=0.7,right]{${f}_p$};
	\draw[red] (3,-1) node[above right,scale=0.7]{$p$} node{$\bullet$};
	\draw[green] (1,-1) node[above left,scale=0.7]{$c$} node{$\bullet$};
	\draw (0.5,-2) .. controls (0.8,-1.5) .. (1,-1) .. controls (2,2) .. (3,-1) .. controls (3.2,-1.7)  .. (3.3,-2); \draw (0,-2.2) node[scale=0.5,right]{4};
	\draw[densely dotted] (0,-1.5) .. controls (0.8,-1.1) and (0.2,-1.7) .. (1,-1) .. controls (2,1.5) and (2.5,-0.5) .. (3,-1) .. controls (3.2,-1.3)  .. (4,-1.6); \draw (-0.5,-1.7) node[scale=0.5,right]{4};
	\draw (2,-3) node[scale=1]{$\C\times \PP^1$};
	
	\draw [<-] (5,0) -- (7,0); \draw (6,-1) node[scale=1]{$\beta_p$};
	
	\draw (8,-1) -- (12,-1); \draw (12,-1) node[scale=0.7,right]{$\tilde{\s}_\infty$};
	\draw[green] (9,-1) node[above left,scale=0.7]{$b$} node{$\bullet$};
	\draw[thick,red] (11.2,-2) -- (12,0.5);\draw[red] (11.3,-2) node[scale=0.7,right]{$\E_p$};
	\draw[thick,blue] (11,2.2) -- (12,0.1);\draw[blue] (11,2.4) node[scale=0.7,right]{$\tilde{f}_p$};
	\draw (8,1) -- (12,1);\draw (7.5,1) node[scale=0.5,right]{0};
	\draw (8,2) -- (12,2);\draw (7.5,2) node[scale=0.5,right]{0};
	\draw[densely dotted] (8,-1.7) .. controls (8.5,-1.5) .. (9,-1) .. controls (10,3.5) and (11,0)  .. (12,-0.8); \draw (7.5,-1.8) node[scale=0.5,right]{3};
	\draw (8.5,-2) .. controls (8.7,-1.8) .. (9,-1) .. controls (9.5,2) and (10,4)  .. (12,-0.4); \draw (8,-2.1) node[scale=0.5,right]{3};		
	\draw (10,-3) node[scale=1]{$\bup_p(\C\times \PP^1)$};
	
	\draw[purple] (11.85,-0.1) node[right,scale=0.7]{$e$} node{$\bullet$};
	\draw[cyan] (11.66,-0.5) node[left,scale=0.7]{$e'$} node{$\bullet$};
	\end{tikzpicture}
\end{center}	

\begin{proof}\
	\begin{enumerate}[wide]
		\item Let $c=\beta_{p}(b)$ and from Lemma \ref{sectiontrivialbundleoverellipticcurve}, there exists a section $s$ of self-intersection $4$ passing with multiplicity $1$ through $c$ and $p$. Then $\s$ be the strict transform of $s$ under $\beta_p^{-1}$, it is a section of self-intersection $3$ passing through $b$ and some point $e\in \E_p\setminus\{\tilde{f}_p,\tilde{\s}_\infty\}$.
		\item Let $U$ be an open neighborhood of  $z$ and $V = U \times (\PP^1 \setminus [0:1])\subset \C\times \PP^1$. For all $\alpha\in \GG_m$ and $\gamma\in \GG_a$, the automorphisms $\phi_{\alpha,\gamma}$ restricted on $V$ gives an isomorphism $V \to \phi_{\alpha,\gamma}(V) $, $\left( x,t \right) \mapsto \left(x,t/(\alpha + \gamma t)\right)$; which extends to an isomorphism $\tilde{\phi}_{\alpha,\gamma}=\beta_p^{-1} \phi_{\alpha,\gamma} \beta_p$ defined on $\bup_pV = \{(x,t),[u:v]\in V\times \PP^1:tu=vf(x)\}$, with $f$ a local parameter of $\O_{\C,z}$, as:
		\begin{align*}
			\tilde{\phi}_{\alpha,\gamma}\colon  \bup_pV & \rightarrow \bup_p(\phi_{\alpha,\gamma}(V)) \\
			\left( (x,t),[u:v] \right)& \mapsto \left( \left( x,\frac{t}{\alpha + \gamma t}\right),\left[u(\alpha+\gamma t):v\right]\right).
		\end{align*}
		In particular, the restriction of $\tilde{\phi}_{\alpha,\gamma}$ on $\E_p$ is obtained by substituting $(x,t)$ by $(z,0)$ and we get:
		\begin{align*}
			\E_p& \rightarrow \E_p \\
			[u:v]&\mapsto [u\alpha:v].
		\end{align*}
		The automorphisms $\phi_{\alpha,\gamma}$ induce an action of $\GG_m$ on the exceptional divisor $\E_p$, with only fixed points $[0:1]$ and $[1:0]$ which correspond to the intersection of $\E_p$ with $\tilde{f}_p$ and $\tilde{\s}_\infty$. So $\GG_m$ acts simply transitively on $\E_p\setminus\{\tilde{f}_p,\tilde{\s}_\infty\}$. In particular if $\s'$ is a section of self-intersection $3$ passing through $e'\in \E_p\setminus \{\tilde{f}_p,\tilde{\s}_\infty\}$ and $b$, then there exists a unique $\alpha\in \GG_m$ and there exists $\gamma \in \GG_a$ such that $\s'$ is the image of $\s$ by $\tilde{\phi}_{\alpha,\gamma}$.
	\end{enumerate}
\end{proof}

\begin{lemma}\label{sectionsa0}
	Under the same notations as in Lemma $\ref{sectionsSp}$, let $p_1\in \E_p\setminus \{\tilde{f}_p,\tilde{\s}_\infty\}$, $\beta_{p_1}\colon \bup_{p_1}(\bup_p(\C\times \PP^1))\to \bup_p(\C\times \PP^1)$ be the blowup of $\bup_p(\C\times \PP^1)$ at $p_1$ and $\beta = \beta_{p_1}\beta_{p}$. Let $\E_{p_1}$ be the exceptional divisor of $\beta_{p_1}$, let $\hat{\E}_p$ and $\hat{\s}_\infty$ be the strict transforms under $\beta_{p_1}^{-1}$ of respectively $\E_p$ and $\tilde{\s}_\infty$. Then we have a simply transitive action of $\GG_a$ on $\E_{p_1}\setminus \hat{\E}_p$ and more precisely for all $d\in \hat{\s}_\infty\setminus \hat{\E}_p$:
	\begin{enumerate}
		\item There exists $e\in \E_{p_1}\setminus \hat{\E}_p$ and a unique section $\sigma$ of self-intersection $2$ passing through $d$ and $e$.
		\item For all $e'\in  \E_{p_1}\setminus \hat{\E}_p$, there exists a unique $\gamma\in \GG_a$ such that $\beta^{-1} \phi_{1,\gamma} \beta(\s)$ is the unique section of self-intersection $2$ passing through $d$ and $e'$.
	\end{enumerate}
\end{lemma}

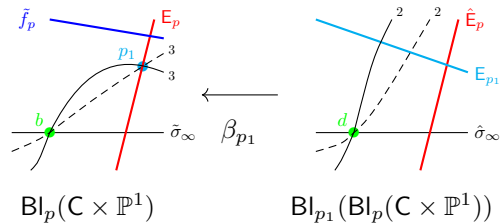
\begin{center}
	\begin{tikzpicture}[scale=0.5]
	\draw (8,-1) -- (12,-1); \draw (12,-1) node[scale=0.7,right]{$\tilde{\s}_\infty$};
	\draw[green] (9,-1) node[above left,scale=0.7]{$b$} node{$\bullet$};
	\draw[cyan] (11.45,0.75) node[above left,scale=0.7]{$p_1$} node{$\bullet$};
	\draw[red,thick] (10.75,-2) -- (11.75,2);\draw[red] (11.75,2) node[scale=0.7,right]{$\E_p$};
	\draw[blue,thick] (9,2) -- (12,1.5);\draw[blue] (8,2) node[scale=0.7,right]{$\tilde{f}_p$};
	\draw (8.5,-2) .. controls (8.7,-1.8) .. (9,-1) .. controls (10,1) and (11,1)  .. (12,0.6); \draw (12,0.5) node[scale=0.5,right]{3};		
	\draw[densely dashed] (8,-1.5) .. controls (8.7,-1.2) .. (9,-1) .. controls (10,-0.2) and (11,0.5)  .. (12,1.1); \draw (12,1.2) node[scale=0.5,right]{3};			
	\draw (10,-3) node[scale=1]{$\bup_p(\C\times \PP^1)$};
	
	\draw[<-] (13,0) -- (15,0);\draw (14,-1) node[scale=1]{$\beta_{p_1}$};
	
	\draw (16,-1) -- (20,-1); \draw (20,-1) node[scale=0.7,right]{$\hat{\s}_\infty$};
	\draw[red,thick] (18.75,-2) -- (19.75,2);\draw[red] (19.75,2) node[scale=0.7,right]{$\hat{\E}_p$};
	\draw[cyan,thick] (16,2) -- (20,0.6); \draw (20,0.5) node[scale=0.7,right,cyan]{$\E_{p_1}$};
	\draw[green] (17,-1) node[above left,scale=0.7]{$d$} node{$\bullet$};
	\draw (16.5,-2) .. controls (16.7,-1.7) .. (17,-1) .. controls (17.5,1) .. (18,2); \draw (18,2.2) node[scale=0.5,right]{2};		
	\draw[densely dashed] (16,-1.5) .. controls (16.7,-1.3) .. (17,-1) .. controls (17.5,-0.6) and (18,0.3) .. (19,2); \draw (19,2.2) node[scale=0.5,right]{2};			
	\draw (18,-3) node[scale=1]{$\bup_{p_1}(\bup_p(\C\times \PP^1))$};
	\end{tikzpicture}
\end{center}

\begin{proof}\
	\begin{enumerate}[wide]
		\item Let $b=\beta_{p_1}(d)$ and from Lemma \ref{sectionsSp} $(1)$, there exists a section $s$ of self-intersection $3$ passing through $b$ and $p_1$. Then the strict transform $\s$ of $s$ under $\beta_{p_1}^{-1}$  is a section of self-intersection $2$ passing through $d$ and a point $e\in \E_{p_1}\setminus \hat{\E}_p$.
		\item From Lemma \ref{sectionsSp}, we know $\GG_m$ acts transitively on $\E_p\setminus \{\tilde{f}_p,\tilde{\s}_\infty\}$ so we can assume $p_1$ has coordinates $((z,0),[1:1])$ in $\bup_p(V)$. We choose an open subset $W$ of $\bup_p(V)=\{(x,t),[u:v]\in V\times \PP^1:tu=vf(x)\}$ containing $p_1$ and such that $u\neq 0$. By the change of variable $v\mapsto v/u$, we have $t=vf(x)$ and the isomorphism $\tilde{\phi}_{1,\gamma}$ restricted on $W$ gives:
		 \begin{align}
		 	W &\rightarrow \tilde{\phi}_{1,\gamma}(W) \nonumber\\
		 	\left( (x,vf(x)),[1:v] \right)& \mapsto \left( \left( x,\frac{vf(x)}{1 + \gamma vf(x)}\right),\left[1:\frac{v}{1+\gamma vf(x)}\right]\right). \label{localiso}
		 \end{align}
	Since $W$ is isomorphic to an open subset of $\AA^2$ by the map $(x,v) \mapsto \left((x,(v+1)f(x)),[1:v+1]\right)$ which sends $(z,0)$ to $p_1$, we can rewrite (\ref{localiso}) as:
	\[
	(x,v)  \mapsto \left(x , \frac{v+1}{1+\gamma f(x)(v+1)}-1\right),
	\]
	which sends $(z,0)$ to $(z,0)$. The isomorphism $\tilde{\phi}_{1,\gamma}$ extends to $\bup_{(z,0)}(W)$ by: 
	\[
	\left((x,v),[u_1:u_2]\right) \longmapsto \left(\left(x,\frac{v-\gamma f(x)(v+1)}{1+\gamma f(x) (v+1)}\right),\left[u_1:\frac{u_2-u_1\gamma (v+1)}{1+\gamma f(x)(v+1)}\right] \right),
	\]
	and restricted on $\E_{(z,0)}$ one gets: $[u_1:u_2]\mapsto [u_1:u_2-u_1\gamma]$. In particular $\GG_a$ acts on the exceptional divisor $\E_{(z,0)}$ and the action has a unique fixed point $[0:1]$ which is the intersection of $\E_{p_1}$ with $\hat{\E}_p$. Therefore the action of $\GG_a$ on $\E_{p_1}\setminus \hat{\E}_p$ is simply transitive. In consequence, if $e'\in  \E_{p_1}\setminus \hat{\E}_p$ then there exists a unique $\gamma\in \GG_a$ such that $\beta^{-1} \phi_{1,\gamma} \beta(\s)$ is the unique section of self-intersection $2$ passing through $d$ and $e'$.
	\end{enumerate}

\end{proof}

\begin{lemma}\label{Gaa0}
	Let $\C$ be an elliptic curve and $\s_0$ be the unique minimal section of $\A_{0,\C}$. For all $a\in \s_0$, $b\in \A_{0,\C}\setminus \s_0$ with $a$ and $b$ not in the same fiber, there exists a unique section $\s$ passing through $a$ and $b$ such that $\s^2=2$. Moreover the subgroup $\{\phi_{1,\gamma}\}_{\gamma\in \GG_a}$ of $\Autinfini$ induces a simply transitive action of $\GG_a$ on $f\setminus \s_0$, where $f$ is any fiber of $\A_{0,\C}$.
\end{lemma}	

\begin{proof}
	Let $q=\s_0(\pi(b))$ and $\epsilon_q\colon\A_{0,\C} \dashrightarrow \PP(\O_\C(\pi(q))\oplus \O_\C)$ be an elementary transformation centered on $a$. From Proposition \ref{decomposable} (1), there is a unique point $q_1\in \PP(\O_\C(\pi(q))\oplus \O_\C)$ where all the sections of self-intersection $1$ meet and we have $\epsilon_{q_1}\colon \PP(\O_\C(\pi(q))\oplus \O_\C) \dashrightarrow \C\times \PP^1$. Let $\psi=\epsilon_{q_1}\epsilon_{q}$, then $p=\psi (b)$ belongs to the same constant section as $c=\psi (a)$. Up to an automorphism of $\C\times \PP^1$ we can assume that $c$ and $p$ lie on the infinite section and apply Lemmas \ref{sectionsSp} and \ref{sectionsa0}. Then using the notation of Lemma \ref{sectionsa0}, the contraction $\bup_{p_1}(\bup_p(\C\times \PP^1))\to \A_{0,\C}$ of $\hat{\E}_p$ and of the strict transform $\hat{f}_p$ of $f_p$ is $\GG_a$-equivariant, so there exists a unique section $\s$ of self-intersection $2$ passing through $a$ and $b$. Moreover for all $b'\in f_q\setminus \s_0$, it also follows from Lemma \ref{sectionsa0} that there exists a unique $\gamma\in \GG_a$ such that $\psi^{-1} \phi_{1,\gamma} \psi(\sigma)$ is the unique section of self-intersection $2$ passing through $a$ and $b'$. 
\end{proof}

\begin{proposition} \label{extensionA0}
	Let $\C$ be an elliptic curve. Then $\AutC(\A_{0,\C})$ is isomorphic to $\GG_a$ and the following sequence of algebraic groups is exact:
	\[
	0 \to \GG_a \to \Autzero(\A_{0,\C}) \to \Autzero(\C) \to 0.
	\]
\end{proposition}

\begin{proof}
	From Lemma \ref{Gaa0}, we know that $\GG_a$ is one-to-one to a subgroup of $\AutC(\A_{0,\C})$. Conversely, let $a\in \s_0$, $b\in \A_{0,\C}\setminus \s_0$ with $a$ and $b$ not in the same fiber. Then an automorphism $f\in \AutC(\A_{0,\C})$ sends a section of self-intersection $2$ passing through $a$ and $b$ to a section of self-intersection $2$ passing through $a$ and a point $b'$ in the same fiber as $b$. Let $\psi$, $p$, and $c$ be as defined in the proof of Lemma \ref{Gaa0}, then the automorphism $\psi^{-1} f \psi$ sends a section in $\C\times \PP^1$ of self-intersection $4$ passing through $p$ and $c$ to another section of self-intersection $4$ passing through $p$ and $c$. From Lemma \ref{sectiontrivialbundleoverellipticcurve}, it follows that $\psi^{-1} f \psi=\phi_{1,\gamma}$ for some $\gamma$ and therefore $\AutC(\A_{0,\C})$ is isomorphic to $\GG_a$. Since $\GG_a$ is connected, we get the exact sequence given in the statement. 
\end{proof}

\subsubsection{The algebraic group $\Autzero(\A_{1,\C})$ when $\operatorname{char}(\kk)\neq 2$}\
\vspace{0.2cm}

In this paragraph we assume that the characteristic of $\kk$ is different than $2$. Let $\Delta=\{1,d_1,d_2,d_3\}$ be the subgroup of two torsion points of $\C$ which is isomorphic to $(\ZZ/2\ZZ)^2$. Then $\Delta$ acts on $\C$ by translation and on $\PP^1$ in the following way:
\begin{align*}
	1\colon [u:v] & \mapsto [u:v], \\
	d_1 \colon [u:v] &\mapsto [-u:v], \\
	d_2 \colon [u:v] &\mapsto [v:u], \\
	d_3 \colon [u:v] &\mapsto [-v:u].
\end{align*}
We denote by $\EF$ the quotient of $(\C\times \PP^1)/\Delta$ given by the diagonal action:
\begin{align*}
	\Delta \times(\C\times \PP^1) & \to \C\times \PP^1 \\
	(d_i,(x,[u:v])) & \mapsto (d_i+x,d_i\cdot [u:v]).
\end{align*}

\begin{lemma}\label{rationalfunctionlemma}
	Let $\Delta$ be a finite group acting on a irreducible variety $\X$. Then $\kk(\X)^\Delta$ is isomorphic to $\kk(\X/\Delta)$.
\end{lemma}

\begin{proof}
	Since $\G$ is finite, we can find an affine $\Delta$-invariant open subset $U\subset \X$. Because $\O_\X(U)^\Delta \subset \O_\X(U)$, we have an extension $\operatorname{Frac}(\O_\X(U)^\Delta) \to \operatorname{Frac}(\O_\X(U))^\Delta$. Let $f/g\in \operatorname{Frac}(\O_\X(U))^\Delta$, write $\Delta=\{d_0,d_1,...,d_n\}$ where $d_0$ is the neutral element and consider $g'= \prod_{i=1}^n d_i\cdot g$. Then $f/g = (fg')/(gg')$ and $gg'$ are $\Delta$-invariant, hence $fg'$ as well. Thus $\operatorname{Frac}(\O_\X(U)^\Delta) \simeq \operatorname{Frac}(\O_\X(U))^\Delta$, i.e.\ $\kk(\X/\Delta)$ is isomorphic to $\kk(\X)^\Delta$ by definition of the quotient.
\end{proof}

\begin{lemma}\label{SSfibréPP^1}
  Let $\C$ be an elliptic curve. Then the following hold:
		\begin{enumerate}
			\item The surface $\EF$ is a $\PP^1$-bundle over $\C/\Delta$ with the following structure morphism: 
			\begin{align*}
			\pi:\EF &\longrightarrow \C/\Delta\\
			(x,[u:v]) \bmod\Delta &\longmapsto x \bmod\Delta.
			\end{align*}
		\item If $i:\C\to \C/\Delta$ is the quotient map for the action of $\Delta$ on $\C$ by translations of order $2$, then the pullback bundle $i^{*}(\EF)$ is $\C$-isomorphic to $\C\times \PP^1$.
		\end{enumerate}
\end{lemma}

\begin{proof}\
		\begin{enumerate}[wide]
				\item First one can check that $\pi$ is well-defined. Let $d\colon \C\times \PP^1 \to \EF$ be the quotient map for the diagonal action of $\Delta$ on $\C\times \PP^1$, then the following diagram is commutative:
				\[
				\begin{tikzcd} [row sep = 3em]
				\C\times \PP^1 \arrow[r,"d"]\arrow[d,"\pi_1" left] & \EF \arrow[d,"\pi"] \\
				\C \arrow[r,"i" below]&  \C/\Delta.
				\end{tikzcd}
				\] 
				Every fiber of $\pi$ corresponds to the gluing of $4$ disjoint fibers of $\C\times \PP^1 \to \C/\Delta$. Since every fiber of $\pi$ is isomorphic to $\PP^1$, it follows that $\pi\colon \EF\to \C/\Delta$ is a ruled surface.
				\item Since the diagram in (1) is commutative, there exists $\alpha\colon \C\times \PP^1 \to i^*(\EF)$ such that the following diagram is commutative:
				\[
				\begin{tikzcd} [column sep=2.5em, row sep=3em]
				\C\times \PP^1 \arrow[dr,"\alpha"]\arrow[drr, "d" , bend left] \arrow[ddr, "\pi_1" , bend right,swap]&&\\
				& i^{*}(\EF)  \arrow[r,"p_2"]\arrow[d,"p_1" left] & \EF \arrow[d,"\pi"] \\
				& \C \arrow[r,"i" below]&  \C/\Delta.
				\end{tikzcd}
				\] 
				From Lemma \ref{rationalfunctionlemma}, we have that $\kk(\EF) \simeq \kk(\C\times \PP^1)^{\Delta}$ and hence $\Delta$ is the Galois group of the extension $d^*\colon \kk(\EF) \to \kk(\C\times \PP^1)$ (see e.g.\ \cite[Theorem 4.7]{Jacobson}). In particular, $[\kk(\C\times \PP^1): \kk(\EF)]= \#\Delta = 4$. Since $p_2$ is 4-to-1, it follows that $\alpha^{*}$ is a $\kk$-isomorphism i.e.\ $\alpha$ is a birational morphism. Because $\alpha$ is also bijective, it follows from Zariski's main theorem (see e.g.\ \cite[Corollary 18.12.13]{Grothendieck}) that $\alpha$ is an isomorphism.
		\end{enumerate}
\end{proof}

\begin{lemma}\label{SSAtiyah}
	The ruled surface $\EF$ is isomorphic to an Atiyah bundle over $\C$.
\end{lemma}

\begin{proof}
	Let $i:\C\to \C/\Delta$ and assuming that $\EF$ admits two disjoint sections $\s_1$ and $\s_2$, we will derive a contradiction. For $k\in \{1,2\}$, the pullback sections $i^*\s_k$ defined as:
	\begin{align*}
		\C&\to i^*(\EF) \\
		x & \to (x,\s_k(x\bmod\Delta))
	\end{align*}
	induce two disjoint sections $\alpha^{-1} (i^*\s_k)$ of $\C\times \PP^1$ since $\C\times \PP^1$ is $\C$-isomorphic to $i^*(\EF)$ by Lemma \ref{SSfibréPP^1} (2). Then it implies that $\alpha^{-1} (i^*\s_1)$ and $\alpha^{-1} (i^*\s_2)$ are constant sections. Then for $k\in \{1,2\}$, there exists a constant $[u:v]\in \PP^1$ such that $\alpha^{-1}(i^*\s_k)$ is defined as $\C\to \C\times \PP^1$, $x\mapsto (x,[u:v])$. It implies that $\s_k$ is given by $\C/\Delta \to \EF$, $x\bmod\Delta\mapsto (x,[u:v])\bmod\Delta$, which is not well-defined. Therefore, constant sections of $\C\times \PP^1$ are not obtained by pulling back sections of $\pi$. Thus, there are no disjoint sections of $\pi$ and $\EF$ is an indecomposable $\PP^1$-bundle over $\C/\Delta$.
	Finally $\Delta$ is the kernel of the multiplication by $2$ in $\C$, hence $\C$ is isomorphic to $\C/\Delta$, so $\EF$ is isomorphic to $\A_{0,\C}$ or $\A_{1,\C}$.
\end{proof}

\begin{proposition}\label{extensionA1}
	The following sequence is exact: \[ 0 \to \Delta \to \Autzero(\EF) \to \Autzero(\C/\Delta)\to 0.\] 
	In particular, the ruled surface $\EF$ is $\C$-isomorphic to $\A_{1,\C}$.
\end{proposition}

\begin{proof}
	First we have an injective morphism of algebraic groups $j\colon \Autzero(\C)  \to \Autzero(\EF) $, $t  \mapsto ((x,[u:v])\bmod\Delta \mapsto (t(x),[u:v])\bmod\Delta)$ such that the following diagram commutes:
	\[
	\begin{tikzcd} [column sep=0.1cm]
	\Autzero(\C) \arrow[rd,twoheadrightarrow]\arrow[rr,hookrightarrow,"j"]   && \Autzero(\EF) \arrow[dl,"\pi_*"] \\
	 & \Autzero(\C/\Delta)   . 
	\end{tikzcd}
	\]
	In particular the morphism $\pi_*\colon \Autzero(\EF) \rightarrow \Autzero(\C/\Delta)$ is also surjective. Let $i\colon \C\to \C/\Delta$ then $\ker(\pi_*)$ is a subgroup of $\Aut(i^*(\EF)  )$ by the embedding $f \mapsto (id,f)$. Moreover $i^*(\EF)$ is isomorphic to $\C\times \PP^1$ from Lemma \ref{SSfibréPP^1} (2) and the automorphism $(id,f)$ of $\Autzero(i^*(\EF))$ corresponds to a $\C$-automorphism of $\C\times \PP^1$, i.e.\ of the form $(id,M) $ where $M\in \PGL_2$. For such automorphism to be compatible with the $\Delta$-action, it has to send an orbit to an orbit for the action of $\Delta$ and a direct computation shows that $M$ belongs to one the following matrices:
	\[
	\begin{bmatrix}
	1 & 0 \\
	0 & 1 
	\end{bmatrix},
	\begin{bmatrix}
	-1 & 0 \\
	0 & 1 
	\end{bmatrix},
	\begin{bmatrix}
	0 & 1 \\
	1 & 0 
	\end{bmatrix},
	\begin{bmatrix}
	0 & -1 \\
	1 & 0
	\end{bmatrix};
	\]
	and conversely they all define automorphisms of $\EF$. It follows that $\ker(\pi_*)$ is isomorphic to $\Delta$ and we get the exact sequence in the statement. Since $\Autzero(\EF)$ is a $1$-dimensional algebraic variety and $\EF$ is an Atiyah bundle (Lemma \ref{SSAtiyah}), and we know from Proposition \ref{extensionA0} that $\Autzero(\A_{0,\C})$ is $2$-dimensional algebraic group, it follows from Theorem \ref{Atiyah} that $\EF$ is $\C$-isomorphic to $\A_{1,\C}$.
\end{proof}

\subsubsection{Description of the maximal automorphism groups}
 
 We have reproved the following theorem of Maruyama:
 
\begin{theorem}[\cite{Maruyama} Theorem 3]\label{Maruyamaext}
	Let $\C$ be an elliptic curve. Then for all distinct points $z_1,z_2\in \C$, we have the following exact sequences of algebraic groups:	
	\begin{alignat*}{9}
		0 & \longrightarrow  &&\ \GG_m \ & \longrightarrow &&\ \Autzero(\S_{z_1,z_2}) & \longrightarrow &&\ \C &\ \longrightarrow &&\ 0, \\
		0 & \longrightarrow &&\ \GG_a & \longrightarrow && \Autzero(\A_{0,\C}) \ & \longrightarrow &&\ \C &\ \longrightarrow &&\ 0, \\
		\intertext{Moreover, the surfaces $\Autzero(\S_{z_1,z_2})$ and $\Autzero(\A_{0,\C})$ are commutative algebraic groups and $\Autzero(\A_{0,\C})$ is not isomorphic to a semidirect product $\GG_a \rtimes \Autzero(\C)$. Finally, if the characteristic of $\kk$ is different than $2$ and if $\Delta\simeq (\ZZ/2\ZZ)^2$ denotes the subgroup of $2$-torsion points of $\C$, then the following sequence of algebraic groups is exact:}
		0  & \longrightarrow  &&\ \ \Delta & \longrightarrow && \Autzero(\A_{1,\C}) \ & \longrightarrow &&\ \C &\ \longrightarrow &&\ 0.
	\end{alignat*}
\end{theorem}

\begin{proof}
	The three exact sequences in the statement are proven in Propositions \ref{extensiondecomposable}, \ref{extensionA0}, \ref{extensionA1}. Moreover, the algebraic groups $\Autzero(\A_{0,\C})$ and $\Autzero(\S_{z_1,z_2})$ are commutative from \cite[Corollary 2 p.\ 433]{Rosenlicht}. Because there is no non-trivial morphism from $\C$ to $\Aut(\GG_a) \simeq \kk^*$ (or because $\Autzero(\A_{0,\C})$ is commutative), the algebraic group $\Autzero(\A_{0,\C})$ is not isomorphic to $\GG_a\rtimes \Autzero(\C)$.
\end{proof}

\begin{remark}
	If $\operatorname{char}(\kk)=2$, every elliptic curve is isomorphic to a curve defined by an equation $\C_\mu:=\{x^3 +\mu xyz + yz^2 + y^2z=0\}\subset \PP^2$, for some $\mu\in \kk$ \cite[Appendix A, Proposition 1.3]{Silverman}. Moreover, $\C:=\C_0$ is the unique supersingular elliptic curve and the kernel of $[2]$ is a finite subscheme $\C[2]$ of $\C$. By a tedious computation, one can check that the pullback of $\A_{1,\C}\to \C$ by $[2]$ is isomorphic to $\C\times \PP^1$ as in Lemma \ref{SSfibréPP^1} (2) and $\A_{1,\C}$ that is isomorphic to a quotient $(\C\times \PP^1)/\C[2]$. The kernel of the morphism $\Autzero(\A_{1,\C}) \to \Autzero(\C)$ is trivial as a group, and strictly contains $\C[2]$ as a subscheme.
	
\end{remark}

\begin{remark} \
	\begin{enumerate}[wide]
		\item From \cite[VII.16, Theorem 6]{SerreGA} (see also \cite[Example 1.1.2]{BSU}), the extensions of $\C$ by $\GG_m$ are classified by $\C$ itself. Let $z_1,z_2\in \C$ be distinct points and $\G$ be the $\GG_m$-bundle defined as the complement of the zero section in $\O_\C(z_1-z_2)$. Then we have a morphism $\pi\colon \G \to \C$ with kernel $\GG_m$, i.e.\ an exact sequence:
		\[
		0\to \GG_m \to \G \overset{\pi}{\to} \C \to 0.
		\]
		Let $\S$ be the quotient of $(\G\times \PP^1)$ by $\GG_m$, given by the following action of $\GG_m$ on $\G\times \PP^1$: $t\cdot (g,[u:v]) \mapsto (g\cdot t^{-1}, [g\cdot u,v])$. Then this gives a morphism $\S\to \G/\GG_m\simeq \C$ which endows $\S$ with a structure $\PP^1$-bundle over $\C$. One can check by computing in local charts that $\S\to \C$ is $\C$-isomorphic to $\S_{z_1,z_2}\to \C$, hence the extension $0\to \GG_m \to \Autzero(\S_{z_1,z_2}) \to \C\to 0$ corresponds to the point $z_1-z_2\in \C$.
		\item From \cite[VII. 17, Theorem 7]{SerreGA} (see also \cite[Example 1.1.2]{BSU}), the extensions of $\C$ by $\GG_a$ are classified by $\H^1(\C,\O_\C)\simeq \kk$. Since the extension $0\to \GG_a \to \Autzero(\A_{0,\C}) \to \C \to 0$ does not split, it corresponds to a non zero element of $\kk$.
		\item  Serre shows in \cite[VII. 15, Theorem 5]{SerreGA} that the algebraic groups $\Autzero(\A_{0,\C})$ and $\Autzero(\S_{z_1,z_2})$ are respectively endowed with a canonical structure of $\GG_a$-principal bundle and $\GG_m$-principal bundle over $\C$. From \cite[Theorem 3.(3)]{Maruyama}, the reduced component of $\Autzero(\A_{0,\C})$ is  isomorphic to $\A_{0,\C}\setminus \s_0$, where $\s_0$ is the unique minimal section of $\A_{0,\C}$. From \cite[Theorem 3.(2)]{Maruyama}, the reduced component of $\Autzero(\S_{z_1,z_2})$ is isomorphic to $\S_{z_1,z_2} \setminus \{\s_1,\s_2\}$, where $\s_1,\s_2$ are the two minimal sections of $\S_{z_1,z_2}$. A natural problem is to describe geometrically their group laws, and a formula has been computed explicitly for the group law of $\A_{0,\C}\setminus \s_0$ when $\kk=\mathbb{C}$ in \cite[$\mathsection$ 3.3 p.251]{LorayMarin}. 
	\end{enumerate}
\end{remark}

\subsection{Proof of Theorems \ref{C} and \ref{D}}

\begin{proposition}\label{nonruledcase}\footnote[2]{The idea of the proof is due to Michel Brion.}
	Let $\X$ be a surface and $\G=\Autzero(\X)$. If $\X$ is not birationally equivalent to $\C\times \PP^1$, for some curve $\C$, then $\G$ is an abelian variety and exactly one of the following cases holds:
	\begin{enumerate}
		\item $\G$ is an abelian surface and $\G\simeq \X$.
		\item $\G$ is isomorphic to an elliptic curve and moreover, there exist a not necessarily reduced curve $\Y$ which is connected, a finite subgroup scheme $\F$ of $\G$ and a $\G$-equivariant isomorphism:
		\[\X\simeq (\G\times \Y)/\F.\]
		The quotient $(\G\times \Y)/\F$ is given by a diagonal action of $\F$ over $\G\times \Y$, $f\cdot (g,y) \mapsto (g\cdot f^{-1},f\cdot y)$. 
		\item $\G$ is trivial.
	\end{enumerate}
	In case $(2)$, if the characteristic of $\kk$ is zero then $\sf{F}$ is reduced and $\Y$ is smooth.
\end{proposition}

\begin{proof}
	From Chevalley's structure theorem (see e.g.\ \cite[Theorem 1.1.1]{BSU}), there exists an exact sequence $0\to \L \to \G \to \sf{A}\to 0$ where $\L$ is a linear algebraic group and $\sf{A}$ is an abelian variety. If $\L$ is not trivial, it contains a $\GG_a$ or a $\GG_m$ and this implies that $\X$ is birationally equivalent to $\C\times \PP^1$ for some curve $\C$ (see \cite[Proposition 2.5.4]{BFT}). Thus $\L$ is trivial, i.e.\ $\G$ is isomorphic to $\sf{A}$.
	
	First suppose that $\G$ has an open orbit $O$ in $\X$ which is isomorphic to $\G/\operatorname{Stab}(x)$ for some $x\in O$. Since $\G$ is commutative and acts faithfully on $O$, it follows that $\operatorname{Stab}(x)$ is trivial and hence $O\simeq \G$. Because $O$ is also the image of the projective morphism $\G\to \X$, $g \mapsto g\cdot x$, then $O$ is closed in $\X$. Therefore, we have $\G\simeq O=\X$ and $\G$ is an abelian surface acting on itself by translation.
	
	Otherwise suppose that $\G$ has an orbit $O$ of dimension $1$ and then for all $x\in O$, the subgroup $\operatorname{Stab}(x)$ is finite (see \cite[Proposition 2.2.1]{BSU}). Therefore $\G$ is an elliptic curve. From \cite[Theorem 2.2.2 and the paragraph following the theorem]{BSU}, there exist a positive integer $n$ and a $\G$-equivariant isomorphism $\h\colon \X\to (\G\times \tilde{\Y})/\G_n$, where $\G_n$ denotes the finite subgroup scheme of $n$-torsion points of $\G$ and $\tilde{\Y}$ is a closed subscheme of $\X$ of dimension $1$. The projection to the first factor $\G\times \tilde{\Y}\to \G$ induces a morphism $g\colon (\G\times \tilde{\Y})/\G_n \to \G/\G_n$ which is $\G$-equivariant. The Stein factorization of $f= gh$ gives morphisms $u\colon \X\to \sf{Z}$ and $v\colon \sf{Z} \to \G/\G_n$, such that $u$ has connected fibers, $v$ is finite and $f=vu$. From Blanchard's Lemma, there exists an action of $\G$ on $\sf{Z}$ such that $u$ is $\G$-equivariant. Since $u$ is also surjective and $f$ is $\G$-equivariant, it follows that $v$ is $\G$-equivariant. Moreover, $\sf{Z}$ is a curve because $v$ is finite, hence it is an orbit of the $\G$-action and the stabilizer $\sf{F}$ of a point is finite. Therefore $\sf{Z}\simeq \G/\sf{F}$ and $u\colon \X\to \sf{Z}\simeq \G/\sf{F}$ is a $\G$-equivariant morphism. From \cite[Section 2.5, paragraph following Lemma 2.10]{Brion} (see also \cite[Paragraph following Example 6.1.2]{BSU}), $\X$ is isomorphic to the $\sf{F}$-torsor $(\G\times \Y)/\sf{F}$ where $\Y=u^{-1}(\sf{F}/\sf{F})$ is connected and the quotient is given by the diagonal action $f\cdot (g,y) \mapsto (g\cdot f^{-1},f\cdot y)$ for all $g\in\G,f\in \sf{F},y\in \Y$. Since $\X$ is a surface, it implies that $\Y$ is a not necessarily reduced curve in positive characteristic. In characteristic zero, $\G_n$ is reduced because the multiplication by $n$ is an étale endomorphism of $\G$, and it follows that the finite subgroup scheme $\F$ of $\G_n$ is also reduced. Hence $\G\times \Y \to (\G\times \Y)/\sf{F}\simeq \X$ is an étale finite morphism from \cite[Theorem in Section 2.7, p.63]{Mumford}. Because $\X$ is smooth, it follows that $\Y$ is also smooth. 
	
	Finally if the orbits of $\G$ have dimension $0$ then $\G$ is trivial.
\end{proof}

\begin{proof}[\bf{Proof of Theorem \ref{C}}]
Let $\C$ be a curve of genus $\g\geq 1$. From Theorem \ref{A}, not every connected algebraic subgroup of $\Bir(\C\times \PP^1)$ is contained in a maximal one. If $\X$ is rational, then every connected algebraic subgroup of $\Bir(\X)$ is contained in $\Autzero(\PP^2)$ or some $\Autzero(\mathbb{F}_n)$ for $n\neq 1$. If $\X$ is not a ruled surface and is not rational, then Proposition \ref{nonruledcase} implies that $\Autzero(\X)$ is contained in a maximal connected algebraic subgroup of $\Bir(\X)$.
\end{proof}	

\begin{proposition}\label{partialclassification}
	Let $\X$ be a surface over $\kk$ and $\G$ be a maximal connected algebraic subgroup of $\Bir(\X)$. If $\X$ is birationally equivalent to $\C\times \PP^1$ with $\C$ a curve of genus $\g$, then $\G$ is conjugate to one of the following:
	\begin{enumerate}
		\item $\Aut(\PP^2)$ or $\Autzero(\mathbb{F}_n)$ with $n\neq 1$, if $\g=0$.
		\item $\Autzero(\C\times \PP^1)$, or $\Autzero(\A_{0,\C})$, or $\Autzero(\A_{1,\C})$, or $\Autzero(\S_{z_1,z_2})$ for some $z_1,z_2\in \C$, if $\g=1$.
		\item $\Autzero(\C\times \PP^1)$, if $\g\geq 2$.
	\end{enumerate}
	If $\X$ is not birationally equivalent to $\C\times \PP^1$ then up to conjugation we have $\G=\Autzero(\X)$ and one of the following holds:
	\begin{enumerate}
		\item[$(4)$] $\G$ is isomorphic to $\X$, which is an abelian surface.
		\item[$(5)$] $\G$ is isomorphic to an elliptic curve and moreover, there exist a not necessarily reduced curve $\Y$ which is connected, a finite subgroup scheme $\F$ and a $\G$-equivariant isomorphism:
		\[\X\simeq (\G\times \Y)/\F.\]
		The quotient $(\G\times \Y)/\F$ is given by a diagonal action of $\F$ over $\G\times \Y$, $f\cdot (g,y) \mapsto (g\cdot f^{-1},f\cdot y)$. 
		\item[$(6)$] $\G$ is trivial.
	\end{enumerate}
	In case $(5)$, if the characteristic of $\kk$ is zero then $\sf{F}$ is reduced and $\Y$ is smooth.
\end{proposition}

\begin{proof}
	Let $\X$ be a surface and $\G=\Autzero(\X)$ a maximal algebraic subgroup of $\Bir(\X)$. From Proposition \ref{minimalsurfaces}, $\G$ is conjugate to $\Autzero(\S)$ with $\S$ a minimal surface birationally equivalent to $\X$. 
	
	If $\X$ is birationally equivalent to $\C\times \PP^1$ with $\C$ a curve, it follows that $\S$ is $\PP^2$ or a ruled surface over $\C$ from \cite[Examples V.5.8.2, V.5.8.3 and Remark V.5.8.4]{Hartshorne}. If $\X$ is rational then $\G$ is conjugate to $\Autzero(\PP^2)\simeq \PGL_3$ which is maximal from Lemma \ref{maximal} (3), or to $\Autzero(\mathbb{F}_n)$ for some integer $n\neq 1$. From \cite[$\mathsection 4.2$]{BlancCremona} there exists a surjective group homomorphism $\Autzero(\mathbb{F}_n) \to \PGL_2$, and hence $\Autzero(\mathbb{F}_n)$ is also maximal from Lemma \ref{maximal} (3). If $\X$ is not rational, the statement follows from Theorem \ref{B}. 
	
	Otherwise $\X$ is not birationally equivalent to $\C\times \PP^1$ and the statement follows from Proposition \ref{nonruledcase}.
\end{proof}

\begin{proof}[\bf{Proof of Theorem \ref{D}}]
		Assume that $\kk$ is a field of characteristic $0$, the first two columns of the table are given by the classification of algebraic surfaces. For the last column, the case $\kappa(\X)=-\infty$ follows from Proposition \ref{partialclassification} and Theorem \ref{A}. 

		Assume that $\X$ is a surface isomorphic to $(\C\times \Y)/\F$, where $\C$ is a elliptic curve, $\Y$ is a smooth curve and $\F$ is a finite subgroup of $\Autzero(\C)$ acting diagonally on $\C\times \Y$ (in particular, $\F$ acts on $\C$ by translations). First notice that we have a morphism $\X \to \C/\F$ with fibre $\Y$, because the pullback of $\X\to \C/\F$ by the quotient morphism $\C\to \C/\F$ is $\C\times \Y$. Moreover, the curve $\C/\F$ is an elliptic curve because $\F$ is a finite subgroup of $\Autzero(\C)$. If $\Y\simeq \PP^1$, then $\X$ is a ruled surface over $\C/\F$. If $\Y$ is an elliptic curve, it follows that $\X$ is a quotient of an abelian surface by a finite group. If $\Y/\F \simeq \PP^1$, then $\X$ is a bielliptic surface \cite[Definition VI.19]{Beauville}. Else $\F$ acts on $\Y$ only by translations, then $\F$ is an abelian subgroup of $(\C\times \Y)$ and $\X$ is again an abelian surface.
		If $\Y$ is a smooth curve of general type, then $\kappa(\X) \geq \kappa(\Y)+\kappa(\C/\F)=1$ \cite[Theorem 6.1.1]{Fujino}. Because $\Autzero(\C)$ acts on $\X$ on the left factor, it is an algebraic subgroup of $\Autzero(\X)$ and $\X$ cannot be a surface of general type. Thus $\kappa(\X) = 1$. 
		
		We denote by $E$ the set of all surfaces of the form $(\C\times \Y)/\F$, where $\C$ is a elliptic curve, $\Y$ is a smooth curve of general type and $\F$ is a finite group of $\Autzero(\C)$ acting diagonally on $\C\times \Y$. We have shown that $E$ is included in the set of surfaces of Kodaira dimension $1$. Let $\X'$ be a surface of Kodaira dimension $1$ which is not in $E$, then $\Autzero(\X')$ is trivial by Proposition \ref{partialclassification} (6). If the minimal model $\X$ of $\X'$ is an element of $E$, then $\Autzero(\X')$ is not maximal since it is conjugated to the trivial subgroup of $\Autzero(\X)$. Else, $\Autzero(\X')$ is maximal.
	 	
	 	If $\X$ is an abelian surface, then $\Autzero(\X)\simeq \X$ is maximal. If $\X'$ is not an abelian surface but is birationally equivalent to an abelian surface $\X$, then $\Autzero(\X')$ is trivial by Proposition \ref{partialclassification} (6) (since we have also shown that $\X'$ does not correspond to a surface given by Proposition \ref{partialclassification} (5)), i.e.\ $\Autzero(\X')$ is conjugated to the trivial subgroup of $\Autzero(\X)$ and is not maximal. 
	 	
	 	Let $\X=(\C\times \Y)/\F$ be a bielliptic surface, with $\C,\Y$ elliptic curves, and $\F$ a finite group acting on $\C$ as a group of translations and acting also on $\Y$ not only by translations. Then $\Autzero(\X)\simeq \C$ \cite[section 3]{BennettMiranda} is maximal. Else assume that $\X'$ is not a bielliptic surface but is birational to a bielliptic surface $\X$, then $\Autzero(\X')$ is trivial by Proposition \ref{partialclassification} (6)  and is conjugated to the trivial subgroup of $\Autzero(\X)$. Thus $\Autzero(\X')$ is not maximal.
	 	
	 	From Proposition \ref{partialclassification} (6), if $\X$ is an Enriques surface, or a K3 surface, or a surface of general type, then $\Autzero(\X)$ is trivial. Moreover, it is maximal since $\Autzero(\X')$ is also trivial for any representative $\X'$ of the birational class of $\X$. 
\end{proof}

\begin{remark}\label{obstructionclassificationpositivecar}
	Let $\X$ be a surface and $\G$ be a maximal connected algebraic subgroup of $\Bir(\X)$. 
	\begin{enumerate}[wide]
		\item In positive characteristic, Proposition \ref{partialclassification} (1), (2), (3), (4), (6) still provides pairs $(\X,\Autzero(\X))$ where $\Autzero(\X)$ is maximal. 
		\item In Proposition \ref{partialclassification} (5) is given a list of candidates $\X$ which could satisfy $\Autzero(\X)\simeq \G$. To give the pairs $(\X,\Autzero(\X))$ such that $\Autzero(\X)$ is maximal, it remains to determine for which curve $\Y$ and finite group scheme $\F$ we have $\X\simeq (\G\times \Y)/\F$ and $\Autzero(\X)\simeq \G$. To the extent of the author's knowledge, it is not known if there exists an example of non-reduced curve $\Y$ and non-reduced finite subgroup scheme $\F$ such that the quotient $(\G\times \Y)/\F$ which occurs in Proposition \ref{nonruledcase} (2) and Proposition \ref{partialclassification} (5) is a smooth surface. If such surfaces do exist then the ones such that $\Autzero(\X)$ is an elliptic curve complete the classification of pairs $(\X,\Autzero(\X))$ with $\Autzero(\X)$ maximal in positive characteristic. 
	\end{enumerate}
\end{remark}

\bibliographystyle{alpha} 
\bibliography{bib} 	

\end{document}